\documentclass[11pt]{article}
\usepackage[margin=1in]{geometry}
\usepackage{amsmath,amsthm,amsfonts,amssymb,amscd,overpic,mathrsfs}
\usepackage{dsfont}
\usepackage{amsthm}
\usepackage{amssymb}
\usepackage{graphicx}
\usepackage{hyperref}
\usepackage{enumerate}
\usepackage{xcolor}
\usepackage{amstext,amsthm,amsfonts,amsmath}
\usepackage{enumerate}
\usepackage[utf8]{inputenc}
\usepackage[font=scriptsize]{caption}
\usepackage{indentfirst}
\usepackage{setspace}

\def\RR{{\mathbb{R}}}

\def\NN{{\mathbb{N}}}

\newcommand{\nexto}{\kern -0.54em}

\newcommand{\dZ}{{\cal Z \kern -0.7em Z}}
\newcommand{\dC}{{\rm\hbox{C \kern-0.8em\raise0.2ex\hbox{\vrule height5.4pt width0.7pt}}}}
\newcommand{\dQ}{{\rm\hbox{Q \kern-0.85em\raise0.25ex\hbox{\vrule height5.4pt width0.7pt}}}}

\newcommand{\norm}[1]{||#1||}
\newcommand{\set}[1]{\left\{#1\right\}}
\newcommand{\dist}[2]{\textnormal{dist}\left(#1,#2\right)}
\newcommand{\inpr}[2]{\left\langle#1,#2\right\rangle}
\def\RR{{\mathbb{R}}}

\newcommand{\R}{\mathbb{R}}
\newcommand{\HH}{\mathbb{H}}
\newcommand{\N}{\mathbb{N}}
\def\NN{{\mathbb{N}}}

\newcommand{\dom}[1]{\textnormal{dom}\,#1}
\newcommand{\Ell}{\mathfrak{L}}

\newtheorem{algorithm}{Algorithm}

\newtheorem{theorem}{Theorem}
\newtheorem{lemma}{Lemma}
\newtheorem{proposition}{Proposition}
\newtheorem{definition}{Definition}
\newtheorem{corollary}{Corollary}

\usepackage{float}
\floatstyle{plain}

\begin{document}

\title{ Inexact proximal $\epsilon$-subgradient methods for composite convex optimization problems}
\setcounter{footnote}{1}

\author{
R. Díaz Millán%
\thanks{Federal Institute of Goi\'as, Goiâna,
Brazil} \thanks{IME Federal University of Goi\'as, Goiâna,
Brazil ,
        \href{mailto:rdiazmillan@gmail.com}{\tt rdiazmillan@gmail.com}.}
\and
M. Pentón Machado%
\thanks{Universidade Federal da Bahia, Bahia, BA, Brazil,
        \href{mailto:majentcha@gmail.com}{\tt majentcha@gmail.com}.}}

\date{\small \today}

\maketitle

\begin{abstract}

We present two approximate versions of the proximal subgradient method for minimizing the sum of two convex functions (not necessarily differentiable). At each ite\-ration, the algorithms require inexact evaluations of the proximal operator, as well as, approximate subgradients of the functions (namely: the $\epsilon$-subgradients). The methods use different error criteria for approximating the proximal operators. 
We  provide an analysis of the convergence and rate of convergence properties of these
methods, considering various stepsize rules, including both, diminishing and constant stepsizes.
For the case where one of the functions is smooth, we propose an inexact accelerated version of the proximal gradient method, and prove that the optimal convergence rate for the function values can be achieved. Moreover, we provide some numerical experiments comparing our algorithm with similar recent ones.

\end{abstract}

\medskip

\noindent {\small {\bf AMS Classification (2010):} 65K05, 90C25, 90C30, 49J45.}


\section{Introduction}
\label{sec:int}

Given a Hilbert space $\HH$\, and\, convex\, functions\, $f:\dom f\subseteq \HH \rightarrow \left]-\infty,\infty\right]$ and $g:\dom g\subseteq \HH \rightarrow \left]-\infty,\infty\right]$, not necessarily differentiable, the problem of interest consists of 
\begin{equation}
\label{eq:problem}
\min_{x\in C} \, f(x) + g(x),
\end{equation}
where $C\subseteq \HH$ is a convex and closed set. The wide variety of applications of problem \eqref{eq:problem} 
involving optimization problem, physics, image processing, statistics, engineering, economy, and mathematics, in 
general, can be explored in \cite{bertsekas2014convex, Bert2015, lemarechal_2, Pol87}  and references therein.

As a special case of problem \eqref{eq:problem} we have the extensively studied constrained convex minimization 
problem
\begin{equation}
\label{eq:constrained-prob}
\min_{x\in C} \, f(x)
\end{equation}
(obtained simply by taking $g\equiv0$ in \eqref{eq:problem}).

\emph{Splitting algorithms} for solving optimization problems, variational inequalities, inclusion and 
equilibrium problems have been exhaustively studied in the last years, see \cite{BelloDiaz, bellocruz1, bot, 
diaz, phdthesis, gli, fadili} and references therein. The idea is ancestral: {\it divide et impera} (divide and 
conquer). When we have difficult and structured problems, {\it splitting} is one of the most important and 
systematic techniques for the development of efficient algorithms for solving them. 

Several splitting algorithms have been used to solve problem \eqref{eq:problem}. One of the most famous is the 
\emph{proximal subgradient} method, or \emph{forward-backward} method, which was proposed by Lions and Mercier 
\cite{liomer} and Passty in \cite{PASSTY1979383}, for the more general problem of  finding a zero of the sum of 
two point-to-set operators. 
For the case of solving the constrained minimization problem \eqref{eq:constrained-prob}, the proximal subgradient 
method reduces to the well-known \emph{projected subgradient} method. Subgradient methods were originally developed by Shor and others in the 1970s and, since then, they have been extensively studied due to their applicability to a wide 
variety of problems and efficient use of memory. A basic reference on subgradient methods is Shor's book 
\cite{shor}.
For the general problem \eqref{eq:problem}, the proximal subgradient method has been recently studied, see for 
instance \cite{bertsekas2014convex,Bert2015, Bert2011}\footnote{In \cite{Bert2011} and \cite{Bert2015}, it is 
actually considered incremental proximal subgradient methods for problem \eqref{eq:problem}, when the objective
function is the sum of a large number of convex functions.} for the finite dimensional case, and 
\cite{BelloCruz2016} for infinite dimensional Hilbert spaces. Recently, the reference \cite{tebou} studied the problem through a non-Euclidean proximal distance of Bregman type. 

The proximal subgradient method combines at each iteration a subgradient step of $f$ with a proximal step of $g$. 
The highest drawback in implementing this method is the calculation of the {\it proximal} (or {\it resolvent}) operator, which in the case of \eqref{eq:constrained-prob} is exactly the \emph{projection} operator onto the set $C$. 
The difficulty lies in the fact that the proximal operator is in general not available in exact form or its computation may be very demanding.
Hence, with the view of improving the computational behavior of the proximal subgradient method,
one may consider relaxed procedures to compute this operator.
One of our goals in this work is to relax the proximal subgradient method by allowing approximated evaluations of the resolvent operators, and to analyze the convergence of these modified schemes. 

We point out that there are some works in the literature dealing with inexact calculations of the resolvent operator in proximal subgradient methods. For instance, in \cite{BirMarRay2003} it was introduced an inexact \emph{projected gradient} method for solving \eqref{eq:constrained-prob} and in \cite{schmidt2011} it was presented an inexact \emph{proximal gradient} method for problem \eqref{eq:problem}. However, these works only consider the case where the function $f$ is differentiable.

Another difficult task in applying the proximal subgradient method could be the necessity of finding a subgradient of some complicated non-differentiable function. Hence, in some context, it may be convenient to use an approximation of the subgradient.  
The introduction in the literature of the $\epsilon$-subdifferential concept in \cite{rocke} has been of great 
importance for the development of computationally more efficient algorithms. 
Several algorithms were implemented using this approximation of the subgradient. 
Probably the most studied are the $\epsilon$-subgradient and projected $\epsilon$-subgradient methods, which extend subgradient methods to solve \eqref{eq:constrained-prob}. A large number of articles have been published on this subject, see \cite{burachik, elias, gou, andrea} and references therein.

In this paper, we present two inexact algorithms for solving \eqref{eq:problem}, which will be referred to as \emph{inexact proximal $\epsilon$-subgradient} methods (P$\epsilon$-SMs).
Our schemes follow the ideas of the $\epsilon$-subgradient and inexact \emph{proximal point}
methods \cite{Rock1976,sol_sv_hpe}. Using these two techniques we relax the known schemes in the literature, see 
\cite{BelloCruz2016, bertsekas2014convex, Bert2015} for similar problems. Since our problem is implemented for 
non-smooth functions, the calculus of the subdifferential is substantially enhanced by using the $\epsilon$-subdifferential. 
The inexact P$\epsilon$-SMs proposed here differ basically in the way they calculate an approximation of the 
proximal operator.
The first method uses an \emph{absolute summable} error criterion, which
was introduced in \cite{Rock1976} by Rockafellar for the proximal point algorithm.
Our second method employs the \emph{relative} error criterion proposed in \cite{sol_sv_hpe}, where a \emph{hybrid extragradient proximal point} algorithm was studied. Several rules for the stepsize are presented. 
For both methods we consider constant and diminishing stepsizes 
and also the Polyak stepsize rule, introduced in \cite{Pol87}. These different ways to choose the stepsizes give us interesting behaviors for the sequences generated by the algorithms.

We will show that the inexact P$\epsilon$-SMs have a convergence rate of $\mathcal{O}(1/\sqrt{k})$, which is the same slow speed of convergence exhibited by projected subgradient methods. 
Since Nesterov showed in \cite{Nest2004} that the projected subgradient method is optimal for the general problem \eqref{eq:constrained-prob} (when $f$ is non-differentiable), we cannot expect a better convergence rate for the general problem \eqref{eq:problem}. 
However, if $f$ is differentiable then proximal gradient methods have a convergence rate of $
\mathcal{O}(1/k)$ for solving the composite problem \eqref{eq:problem}.
Furthermore, these methods can be accelerated using Nesterov techniques \cite{nest,fista,Nest2013} to achieve the 
optimal convergence rate of $\mathcal{O}(1/k^2)$.
As discussed before, since the proximal operator may not have an analytic solution, or it may be very expensive to compute exactly, there is an interest in the study of the convergence of accelerated proximal gradient methods with inexact calculation of the resolvent mapping (see \cite{schmidt2011,Villa2013}).

Our second goal in this work is to present an \emph{inexact accelerated proximal gradient} method (PGM) for 
solving \eqref{eq:problem}, and to show that it achieves the optimal convergence rate. 
The main difference between our method and the algorithms proposed in \cite{schmidt2011} and \cite{Villa2013} is the inexactness notion used to evaluate the proximal operator. 
Indeed, we use the relative error criterion proposed in \cite{sol_sv_unif} for \emph{maximal monotone operators}. 
This could bring an advantage in practice when compared to the absolute error criteria proposed in 
\cite{schmidt2011} and \cite{Villa2013}, since in our method we fix a tolerance in advance rather than requiring 
the errors decrease in a certain way. 
Of course, to confirm this assertion it is needed an exhaustive numerical verification with a large sample of problems, and this verification is not the goal of our work. However, in subsection \ref{sub:numer2}, we present preliminary computational experiments which suggest an advantage in using relative error criterion over absolute error criterion.


The manuscript is organized as follows. In Section \ref{prelim} we introduce the notation and some preliminaries that will be used in the remainder of the paper. In Sections \ref{asbt} and Section \ref{relat} we introduce the inexact P$\epsilon$-SMs based on absolute and relative error criteria, respectively. The convergence analyses of these two algorithms with different stepsize rules are presented in Section \ref{converg}. In Section \ref{sec:acel} an inexact accelerated version in the Nesterov's sense is presented. Section \ref{numer} presents numerical examples and comparisons. In Section \ref{sec:con} we present some conclusions.

\section{Preliminaries}\label{prelim}

In this section we describe some classic definitions and results that will be needed along this work. 

First of all, we present the notation.
Throughout this paper, we let $\HH$ denote a Hilbert space with inner product and induced norm denoted by $\inpr{\cdot}{\cdot}$ and $\norm{\cdot}$, respectively. We also define the spaces $\R_{+}$ and $\R_{++}$ as $\R_{+}:=\set{x\in\R\,:\,x\geq0}$ and $\R_{++}:=\set{x\in\R\,:\,x>0}$. We denote by $S^\star$ and $s^\star$ the solution set and optimal value of problem \eqref{eq:problem}, respectively. Given a non-empty convex closed set $C\subset\HH$, we let $P_C(\cdot)$ denote the projection operator onto $C$.

Given an extended real valued convex function $f:\HH\rightarrow\left]-\infty,\infty\right]$, the domain of $f$ is the set $\dom f:=\set{x\in\HH\,:\,f(x)<\infty}$.
Since $f$ is a convex function, it is clear that $\dom{f}$ is convex. We say that function $f$ is \emph{proper} if $\dom f\neq\emptyset$. Furthermore, we say that $f$ is \emph{closed} if it is a lower semicontinuous function.

\begin{definition}
Given a convex function $f:\HH\rightarrow\left]-\infty,\infty\right]$, a vector $v\in\HH$ is called a \emph{subgradient} of $f$ at $x\in\HH$, if
\begin{equation*}
f(x') \geq f(x)+\inpr{v}{x'-x}\qquad\qquad\forall x'\in\HH.
\end{equation*}
The set of all subgradients of $f$ at $x$ is denoted by $\partial f(x)$. The operator $\partial f$, which maps each $x$ to $\partial f(x)$, is called the \emph{subdifferential} map associated with $f$.
\end{definition}

It can be seen immediately from the definition that $x^\star$ is a global minimizer of $f$ in $\HH$ if and only if $0\in\partial f(x^\star)$. The subdifferential mapping of a convex function $f$ has the following \emph{monotonicity} property: for any $x$, $x'$, $v$ and $v'\in\HH$ such that $v\in\partial f(x)$ and $v'\in\partial f(x')$, it follows that
\begin{equation}
\label{eq:monotonicity}
\inpr{x-x'}{v-v'}\geq 0.
\end{equation}

\begin{definition}
Given any convex function $f:\HH\to\left]-\infty,\infty\right]$ and $\epsilon\geq0$, a vector $v\in\HH$ is called an $\epsilon$-\emph{subgradient} of $f$ at $x\in\HH$, if
\begin{equation*}
f(x')\geq f(x)+\inpr{v}{x'-x}-\epsilon\qquad\qquad\forall x'\in\HH.
\end{equation*}
The set of all $\epsilon$-subgradients of $f$ at $x$ is denoted by $\partial_\epsilon f(x)$, and $\partial_\epsilon f$ is called the $\epsilon$-\emph{subdifferential} mapping.
\end{definition}

It is trivial to verify that $\partial_0f(x)=\partial f(x)$, and $\partial f(x)\subseteq\partial_\epsilon f(x)$ for every $x\in\HH$ and $\epsilon\geq0$. The proposition below lists some useful properties of the $\epsilon$-subdifferential that will be needed in our presentation.

\begin{proposition}
\label{prop:sub-properties}
If $f:\HH\to\left]-\infty,\infty\right]$ and $g:\HH\to\left]-\infty,\infty\right]$ are proper closed convex functions,  then the following statements hold.
\begin{itemize}
\item[(i)] $\partial_{\epsilon_1} f(x) + \partial_{\epsilon_2}g(x)\subset\partial_{\epsilon_1+\epsilon_2}(f+g)(x)$ for all $x\in\HH$ and $\epsilon_1$, $\epsilon_2\geq0$.
\item[(ii)] If $v\in\partial f(x)$ and $f(z)<\infty$, then $v\in\partial_\epsilon f(z)$ where $\epsilon:=f(z)-f(x)-\inpr{v}{z-x}$.
\end{itemize}
\end{proposition}

\begin{proof}
Statements (i) and (ii) are classical results which can be found, for example, in \cite{lemarechal_2}. 
\end{proof}

Given $\alpha>0$, the \emph{proximal operator} (or \emph{resolvent operator}) \cite{moreau} associated with $\partial f$ is defined as
\begin{equation}
\label{eq:prox}
\textbf{prox}_{\alpha f}(z): = \arg\min_{x\in\HH} \big\{\alpha f(x) + \frac{1}{2}\norm{x-z}^2\big\},\qquad\qquad\forall z\in\HH.
\end{equation}
The fact that $\textbf{prox}_{\alpha f}:\HH\to\HH$ is an everywhere well defined function, if $f$ is a proper closed convex function, is a consequence of a fundamental result due to Minty \cite{minty}.

For evaluating $\textbf{prox}_{\alpha f}$ it is
necessary to solve a strongly convex minimization
problem. Since this problem could be hard to solve exactly, for instance, if
$f$ has a complicated algebraic expression, it is desirable to allow approximate
evaluations of the proximal operator. Several notions of inexactness of the resolvent operator
have been proposed in the literature. In this work we
are interested in three different criteria for  
approximating $\textbf{prox}_{\alpha f}$. The first general
criterion we treat here 
was proposed in \cite{Rock1976} by Rockafellar for the 
proximal point algorithm for maximal monotone operators.

\begin{definition}
\label{def:r-approx}
Given $r>0$, a point $x\in\HH$ is called an $r$-\emph{approximate solution} of $\textbf{prox}_{\alpha f}$ at $z$, if there exists $v\in\partial f(x)$ such that 
$$\norm{\alpha v + x-z}\leq r.$$
\end{definition}

We note that if $x$ is an $r$-approximate solution of $\textbf{prox}_{\alpha f}$ at $z$, then the distance between $x$ and $\textbf{prox}_{\alpha f}(z)$ is less than $r$. Indeed, defining $y=\textbf{prox}_{\alpha f}(z)$, by the first-order optimality condition of the minimization problem in \eqref{eq:prox}, we have that there exists $w \in\partial f(y)$ such that $\alpha w + y = z$. Thus,
\begin{equation}
\label{eq:nova}
\begin{split}
\norm{\alpha v + x - z}^2 = & \norm{\alpha v + x - y - \alpha w}^2 \\
= &\norm{x - y}^2 + 2\alpha\inpr{v - w}{x - y} + \norm{\alpha(v - w)}^2\\
\geq & \norm{x - y}^2,
\end{split}
\end{equation}
where the inequality above is due to the monotonicity property of the subdifferential map $\partial f$ (see equation \eqref{eq:monotonicity}). Therefore, we conclude that $\norm{x - \textbf{prox}_{\alpha f}(z)}\leq r$.

The other criteria we employ for approximating \eqref{eq:prox} were introduced in 
\cite{sol_sv_hpe} and \cite{sol_sv_unif}, where inexact versions of the proximal point algorithm were proposed. In this work, we actually use optimization versions of these \emph{relative} error criteria, where an inclusion in terms of the $\epsilon$\emph{-enlargement} of a maximal monotone operator is replaced by an inclusion in terms of the $\epsilon$-subdifferential of a proper closed convex function.

Before presenting these notions of inexact solution we observe that evaluating $\textbf{prox}_{\alpha f}(z)$ is equivalent to solving the \textit{proximal system}: find a pair $x,\,v\in\HH$ such that
\begin{equation}
\label{eq:proximal-system}
\left\lbrace
\begin{array}{l}
v \in\partial f(x),\\
\alpha v + x -z = 0.
\end{array}
\right.
\end{equation}


\begin{definition}\emph{\cite{sol_sv_hpe}}
\label{def:approx-sol}
Given $\sigma\in\left[0,1\right[$, a triplet $(x,v,\epsilon)\in\mathbb{H}\times\HH\times\R_{+}$ is called a $\sigma$\emph{-approximate solution} of \eqref{eq:proximal-system} at $(\alpha,z)$, if
\begin{equation}
\label{eq:approx-sol}
\begin{split}
& v\in \partial_{\epsilon} f(x),\\
& \norm{\alpha v + x - z}^{2} + 2\alpha\epsilon \leq \sigma^2\norm{x-z}^{2}.
\end{split}
\end{equation}
\end{definition}

\begin{definition}\emph{\cite{sol_sv_unif}}
\label{def:approx-sol2}
Given $\sigma\in\left[0,1\right[$, a triplet $(x,v,\epsilon)\in\mathbb{H}\times\HH\times\R_{+}$ is called a $\sigma$\emph{-quasi-approximate solution} of \eqref{eq:proximal-system} at $(\alpha,z)$, if
\begin{equation}
\label{eq:approx-sol2}
\begin{split}
& v\in \partial_{\epsilon} f(x),\\
& \norm{\alpha v + x - z}^{2} + 2\alpha\epsilon \leq \sigma^2(\norm{\alpha v}^2 + \norm{x-z}^{2}).
\end{split}
\end{equation}
\end{definition}

We observe that if $(x,v)$ is the exact solution of (\ref{eq:proximal-system}), then taking $\epsilon=0$ the triplet $(x,v,\epsilon)$ satisfies the approximation criteria
(\ref{eq:approx-sol}) and \eqref{eq:approx-sol2} for all $\sigma\in\left[0,1\right[$. Conversely, if $\sigma=0$ only the exact solution of (\ref{eq:proximal-system}), taking $\epsilon=0$, will satisfy (\ref{eq:approx-sol}) and \eqref{eq:approx-sol2}. For $\sigma\in\left]0,1\right[$ system \eqref{eq:proximal-system} has at least one, and typically many, approximate solutions in the sense of Definitions \ref{def:approx-sol} and \ref{def:approx-sol2}.

The convergence analysis of a wide range of optimization algorithms relies on the notions of quasi-Fejér and Fejér convergence, which were originated in \cite{QuasiFejer69} in the context of sequences of random 
variables. We finish this section by introducing these concepts, and two important properties of quasi-Fejér 
convergent sequences that will be needed in our presentation. For more insight on this subject we refer the 
reader to \cite{Comb001}.

\begin{definition}
\label{def:quasi-fejer}
Let $S$ be a non-empty set of $\HH$. A sequence $(x^k)_{k\in\N}$ is quasi-Fejér convergent to $S$ if for all $x\in S$ there is a sequence $(r_k)_{k\in\N}\in\ell_1(\R)$  such that $\norm{x^{k+1}-x}^2\leq\norm{x^k-x}^2+r_k$, for all integer $k\geq0$. If $r_k=0$ for all $k\in\N$, we say that $(x^k)_{k\in\N}$ is Fejér convergent to $S$.
\end{definition}

\begin{proposition}
\label{prop:fejer}
If the sequence $(x^k)_{k\in\N}$ is quasi-Fejér convergent to $S$, then
\begin{itemize}
\item[(i)] $(x^k)_{k\in\N}$ is bounded.
\item[(ii)] $(x^k)_{k\in\N}$ converges weakly to some point in $S$ if and only if all weak accumulation points of  $(x^k)_{k\in\N}$ belong to $S$.
\end{itemize}
\end{proposition}

\section{An inexact P$\epsilon$-SM with absolute error}\label{asbt}
In this section, we propose an inexact P$\epsilon$-SM for solving problem \eqref{eq:problem} that uses absolute summable error criteria. 
Algorithm \ref{alg:1} below generalizes the proximal subgradient method by allowing, at each iteration $k$, $\textbf{prox}_{\alpha_kg}$ to be approximately evaluated, where $\alpha_k>0$, provided that the inexact solution satisfies Definition \ref{def:r-approx} for a predefined $r_k$. This algorithm also includes the use of $\epsilon$-subgradients of function $f$ rather than just subgradients. 

In what follows we assume that:
\begin{enumerate}
\item[(A.1)] $\dom{g} \subseteq \dom{f}$,
\item[(A.2)] $C$ is such that $C \subseteq \text{int}(\dom{g})$.
\end{enumerate}

\fbox{\parbox{14.3cm}{
\begin{algorithm}
\label{alg:1}
Pick an initial point $x^0\in C$, a square summable sequence $(r_k)_{k\in\N}\subset\R_+$, and sequences $(\alpha_{k})_{k\in \mathbb{N}}\subset \mathbb{R}_{++}$ and $(\epsilon_{k})_{k\in \mathbb{N}}\subset \mathbb{R}_{+}$. For all $k=0,1,\dots$
\begin{enumerate}
\item Compute
\begin{equation*}
y^k = x^k - \alpha_ku^k, \qquad \textrm{where }\,\, u^k\in\partial_{\epsilon_k} f(x^k).
\end{equation*}
\item Compute $\overline{x}^k$ and $\overline{w}^k$ such that $\overline{w}^k\in\partial g(\overline{x}^k)$ and
\begin{equation*}
\norm{\alpha_k\overline{w}^k + \overline{x}^k - y^k} \leq r_k.
\end{equation*}
\item Set
\begin{equation*}
x^{k+1} = P_C(y^k - \alpha_k\overline{w}^k).
\end{equation*}
\end{enumerate}
\end{algorithm}}}

\vspace{2mm}

Several remarks are in order. 
First, we note that Algorithm \ref{alg:1} is well defined since $x^k\in C$ for every $k\geq0$, therefore $\partial f(x^k)\neq\emptyset$ by assumptions A.1 and A.2.

Second, we observe that Proposition \ref{prop:sub-properties}(ii) implies, for all $k\in\N$, that $\overline{w}^k\in\partial_{\overline{\epsilon}_k}g(x^k)$, where $\overline{\epsilon}_k:=g(x^k)-g(\overline{x}^k)-\inpr{\overline{w}^k}{x^k-\overline{x}^k}$. Furthermore, by Proposition \ref{prop:sub-properties}(i) we have that $u^k + \overline{w}^k\in\partial_{\epsilon_k+\overline{\epsilon}_k}(f+g)(x^k)$. Thus, Algorithm \ref{alg:1} can be seen as an $\epsilon$-subgradient projected method \cite{bertsekas2014convex} applied to the sum $f+g$. However, we stress here that our convergence analysis for Algorithm \ref{alg:1} does not use this relation.

Third, we note that Algorithm \ref{alg:1} actually recovers the proximal subgradient methods introduced in \cite{BelloCruz2016,Bert2015,Bert2011} for solving problem \eqref{eq:problem}. Indeed, assuming that $C=\dom{g}$, if we take $r_k=0$ and $\epsilon_k=0$ for all $k\in\N$, then we have that $u^k\in\partial f(x^k)$ and step  2 above implies $\alpha_k\overline{w}^k+\overline{x}^k-y^k=0$. Hence, $\overline{x}^k=\textbf{prox}_{\alpha_kg}(y^k)$ and step 3 in Algorithm \ref{alg:1} yields $x^{k+1}=\overline{x}^k$. Thus, steps 1-3 above reduce exactly to the set of recursions of the methods in \cite{Bert2015} (when $m=1$) and \cite{BelloCruz2016}.

We now present two technical results that will be needed in our convergence analysis in section \ref{sec:conv-ana}. Lemma \ref{lem:2} below establishes an important feature of the sequence $(x^k)_{k\in\N}$
generated by Algorithm \ref{alg:1}.
It is similar in spirit to key results on which the convergence of subgradient and proximal subgradient methods relies (see \cite[Lemma 2.1]{BelloCruz2016}, \cite[Proposition 3.1]{Bert2015} and \cite[Proposition 6.3.1]{bertsekas2014convex}).

\begin{lemma}
\label{lem:2}
If $(x^k)_{k\in\N}$, $(u^k)_{k\in\N}$ and $(\overline{w}^k)_{k\in\N}$ are the sequences generated by Algorithm $\ref{alg:1}$, then, for all $k\in\N$ and $x\in C$, we have

\begin{equation}
\label{eq:sta-lema1}
\norm{x^{k+1}-x}^2 \leq \norm{x^k-x}^2 + 2\alpha_k((f+g)(x) - (f+g)(x^k) + \epsilon_k) + r_k^2 + \alpha_k^2\norm{u^k+w^k}^2
\end{equation}
for any $w^k\in\partial g(x^k)$.
\end{lemma}

\begin{proof}
Since $x^{k+1}=P_C(y^k-\alpha_k\overline{w}^k)$, it follows from the non-expansion property of the projection operator $P_C(\cdot)$ that, for any $x\in C$
\begin{equation}
\label{eq:eq0}
\norm{x^{k+1}-x} \leq \norm{y^k-\alpha_k\overline{w}^k-x}.
\end{equation}
Now, we note that
\begin{equation}
\label{eq:11}
\begin{split}
\norm{y^k - \alpha_k\overline{w}^k - x}^2 = &  \norm{y^k - x}^2 - 2\alpha_k\inpr{\overline{w}^k}{y^k-x} + \alpha_k^2\norm{\overline{w}^k}^2\\
= & \norm{y^k-x}^2 - 2\alpha_k\inpr{\overline{w}^k}{\overline{x}^k-x} - 2\alpha_k\inpr{\overline{w}^k}{y^k-\overline{x}^k} + \alpha_k^2\norm{\overline{w}^k}^2\\
 = & \norm{y^k-x}^2 - 2\alpha_k\inpr{\overline{w}^k}{\overline{x}^k-x} + \norm{\alpha_k\overline{w}^k + \overline{x}^k-y^k}^2
- \norm{y^k-\overline{x}^k}^2,
\end{split}
\end{equation}
where the last equality above follows from a simple manipulation.
Moreover, using the definition of $y^k$ in step 1 of Algorithm \ref{alg:1} we have
\begin{equation*}
\norm{y^k-x}^2 = \norm{x^k-x}^2 -2\alpha_k\inpr{u^k}{x^k-x} + \alpha_k^2\norm{u^k}^2.
\end{equation*}
Hence, combining equation above with \eqref{eq:11} and rearranging the terms, we obtain
\begin{equation}
\label{eq:1}
\begin{split}
\norm{y^k - \alpha_k\overline{w}^k - x}^2 = & \norm{x^k-x}^2 - 2\alpha_k\inpr{u^k}{x^k-x} - 2\alpha_k\inpr{\overline{w}^k}{\overline{x}^k-x}  \\
& + \norm{\alpha_k\overline{w}^k + \overline{x}^k-y^k}^2
- \norm{y^k-\overline{x}^k}^2 + \alpha_k^2\norm{u^k}^2.
\end{split}
\end{equation}

Next, we observe that assumptions $u^k\in\partial_{\epsilon_k} f(x^k)$ and   $\overline{w}^k\in\partial g(\overline{x}^k)$ imply, respectively, that
\begin{equation}
\label{eq:sub-fg}
\inpr{u^k}{x-x^k} \leq f(x) - f(x^k) + \epsilon_k\qquad\text{ and }\qquad \inpr{\overline{w}^k}{x-\overline{x}^k} \leq g(x) - g(\overline{x}^k).
\end{equation}
Thus, equation \eqref{eq:eq0}, together with \eqref{eq:1} and \eqref{eq:sub-fg}, yields
\begin{equation*}
\begin{split}
\norm{x^{k+1}-x}^2 \leq & \norm{x^k-x}^2 + 2\alpha_k(f(x) - f(x^k) + \epsilon_k) + 2\alpha_k (g(x) - g(\overline{x}^k))\\
& + \norm{\alpha_k\overline{w}^k + \overline{x}^k-y^k}^2 - \norm{y^k-\overline{x}^k}^2 + \norm{\alpha_ku^k}^2.\\
\end{split}
\end{equation*}
We now combine the equation above with the error criterion in step 2 of Algorithm \ref{alg:1} and the relation
\begin{equation}
\label{eq:3}
\begin{split}
\norm{\alpha_ku^k}^2-\norm{y^k-\overline{x}^k}^2 = \norm{\alpha_ku^k}^2 - \norm{x^k-\alpha_ku^k-\overline{x}^k}^2
= 2\alpha_k\inpr{u^k}{x^k-\overline{x}^k} - \norm{x^k-\overline{x}^k}^2,
\end{split}
\end{equation}
to deduce that 
\begin{equation*}
\begin{split}
\norm{x^{k+1}-x}^2 \leq & \norm{x^k-x}^2 + 2\alpha_k(f(x) - f(x^k) + \epsilon_k + g(x) - g(\overline{x}^k)) + r_k^2\\
& + 2\alpha_k\inpr{u^k}{x^k-\overline{x}^k} - \norm{x^k-\overline{x}^k}^2.
\end{split}
\end{equation*}

Adding and subtracting $g(\overline{x}^k)$ on the right-hand side term of the equation above, and taking $w^k\in \partial g(x^k)$, which exists due to the fact that $x^k\in C$ and hypothesis A.2, we obtain
\begin{equation*}
\begin{split}
\norm{x^{k+1}-x}^2 \leq & \norm{x^k-x}^2 + 2\alpha_k((f+g)(x) - (f+g)(x^k) + \epsilon_k) + 2\alpha_k(g(x^k)-g(\overline{x}^k)) \\
& + r_k^2 + 2\alpha_k\inpr{u^k}{x^k-\overline{x}^k} - \norm{x^k-\overline{x}^k}^2\\
\leq & \norm{x^k-x}^2 + 2\alpha_k((f+g)(x) - (f+g)(x^k) + \epsilon_k) + 2\alpha_k\inpr{w^k}{x^k-\overline{x}^k}\\
& + r_k^2 + 2\alpha_k\inpr{u^k}{x^k-\overline{x}^k} - \norm{x^k-\overline{x}^k}^2.
\end{split}
\end{equation*}

Finally, by straightforward calculations in the last inequality above we deduce that
\begin{equation*}
\begin{split}
\norm{x^{k+1}-x}^2 \leq &
\norm{x^k-x}^2 + 2\alpha_k((f+g)(x) - (f+g)(x^k)+\epsilon_k) + r_k^2\\
& + \norm{\alpha_k(w^k+u^k)}^2 - \norm{x^k-\overline{x}^k-\alpha_k(w^k+u^k)}^2,
\end{split}
\end{equation*}
which implies equation \eqref{eq:sta-lema1}.
\end{proof}

It is to be expected that Algorithm \ref{alg:1} is not a descent method since it recovers the classical subgradient method, and it is well known that an iteration of this method can increase the objective value of the function. Therefore, as in subgradient methods, we will keep track of the best point found so far, i.e. the one with the smallest function value.
For this purpose, we define recursively the sequence of functional values $((f+g)^k_\text{best})_{k\in\N}$ as
\begin{equation}
\label{eq:best-func}
(f+g)^k_\text{best}: =
\left\lbrace
\begin{array}{ll}
(f+g)(x^0),&\quad\text{if }\, k = 0,\\
\min\{(f+g)^{k-1}_\text{best}\,,\,(f+g)(x^k)\},&\quad\text{for all }\, k\geq1.
\end{array}\right.
\end{equation}
We observe that $((f+g)^k_\text{best})_{k\in\N}$ is a decreasing sequence, therefore it has a limit, which can be $-\infty$.

The following result is a consequence of Lemma \ref{lem:2} and it presents a general convergence rate for the sequence of the best objective values \eqref{eq:best-func}.

\begin{lemma}
\label{lem:conv-rate-func-alg1}
Let $(x^k)_{k\in\N}$ be the sequence generated by Algorithm $\ref{alg:1}$ and $((f+g)^k_\text{best})_{k\in\N}$ be the sequence defined in \eqref{eq:best-func}. If $S^\star\neq\emptyset$, then, for all integer $k\geq0$, it holds that
\begin{equation}
\label{eq:conv-rate-alg1}
(f+g)^k_\text{best} - s^\star \leq \dfrac{d_0^2 + 2\sum_{j=0}^k\alpha_j\epsilon_j + \sum_{j=0}^kr_j^2 + c_k\sum_{j=0}^k\alpha_j^2}{2\sum_{j=0}^k\alpha_j},
\end{equation}
where $d_0$ is the distance of $x^0$ to $S^\star$, $c_k:=\max\limits_{j=0,\dots,k}\left\lbrace\norm{u^j+w^j}^2\right\rbrace$ and $w^j\in\partial g(x^j)$.
\end{lemma}

\begin{proof}
We define $x^\star := P_{S^\star}(x^0)$ and take $x=x^\star$ in Lemma \ref{lem:2} to obtain
\begin{equation*}
\begin{split}
\norm{x^{k+1}-x^\star}^2 \leq &  \norm{x^k-x^\star}^2 - 2\alpha_k((f+g)(x^k)-s^\star) + 2\alpha_k\epsilon_k + r_k^2 + \alpha_k^2\norm{u^k+w^k}^2.
\end{split}
\end{equation*}
Now, applying recursively the equation above we have
\begin{equation*}
\begin{split}
\norm{x^{k+1}-x^\star}^2 \leq & \norm{x^0-x^\star}^2 - 2\sum_{j=0}^k\alpha_j((f+g)(x^j)-s^\star) + 2\sum_{j=0}^k\alpha_j\epsilon_j\\
&+ \sum_{j=0}^kr_j^2 + \sum_{j=0}^k\alpha_j^2\norm{u^j+w^j}^2.
\end{split}
\end{equation*}
Since $d_0:=\norm{x^0-x^\star}$, inequality \eqref{eq:conv-rate-alg1} follows from a simple manipulation of the equation above and the definitions of $(f+g)^k_\text{best}$ and $c_k$.
\end{proof}

\section{An inexact P$\epsilon$-SM with relative error}\label{relat}

In this section, we derive an inexact algorithm for solving \eqref{eq:problem} with theoretical bases similar to Algorithm \ref{alg:1}, but with a relative error criterion. Specifically, Algorithm \ref{alg:2} below differs from the inexact P$\epsilon$-SM proposed in the previous section in the way that the approximation of the proximal operator is calculated.
In the following algorithm, it is used the notion of approximate solution of a resolvent operator presented in Definition \ref{def:approx-sol}.
Also, the $\epsilon$-subdifferential is present in both functions $f$ and $g$, in contrast with Algorithm \ref{alg:1} which only considers $\epsilon$-subgradients of $f$.

\vspace{2mm}

\fbox{\parbox{14.3cm}{
\begin{algorithm}
\label{alg:2}
Pick $\sigma\in\left[0,1\right[$, an initial point $x^0\in C$  and sequences $(\alpha_{k})_{k\in \mathds{N}}\subset \mathds{R}_{++}$ and $(\epsilon_{k})_{k\in \mathds{N}}\subset \mathds{R}_{+}$. For $k=0,1,\dots$
\begin{enumerate}
\item Take $u^k\in\partial_{\epsilon_k} f(x^k)$ and set
\begin{equation*}
y^k = x^k - \alpha_k u^k.
\end{equation*}
\item Compute a triplet $(\overline{x}^k,\overline{w}^k,\overline{\epsilon}_k)\in\HH\times\HH\times\R_+$   such that $\overline{w}^k\in\partial _{\overline{\epsilon}_k}g(\overline{x}^k)$ and
\begin{equation*}
\norm{\alpha_k\overline{w}^k + \overline{x}^k - y^k}^2 + 2\alpha_k\overline{\epsilon}_k \leq \sigma^2\norm{\overline{x}^k-y^k}^2.
\end{equation*}
\item Set
\begin{equation*}
x^{k+1} = P_C(y^k - \alpha_k\overline{w}^k).
\end{equation*}
\end{enumerate}
\end{algorithm}}}

\vspace{2mm}

We observe that the methods of \cite{BelloCruz2016,Bert2015,Bert2011} are also particular cases of Algorithm 
\ref{alg:2}.
Indeed, by choosing $C=\dom{g}$, $\sigma=0$ and $\epsilon_k=0$ for all $k\in\N$, it can be seen that steps 1-3 above 
are precisely the iterations of the proximal subgradient methods in \cite{BelloCruz2016} and \cite{Bert2015}
(when $m=1$).

We also note that Algorithm \ref{alg:2} does not specify how to find $\overline{x}^k$, $\overline{w}^k$ and $\overline{\epsilon}_k$ as in step 2. The algorithm used to perform this computation  will depend on the 
particular implementation of the method and the properties of the operator $\partial g$.

Next result will be fundamental for developing the convergence analysis of Algorithm \ref{alg:2}.
It is of the same nature as Lemma \ref{lem:2} and may be proven in much the same way.

\begin{lemma}
\label{lem:1}
Let $(x^k)_{k\in\N}$, $(u^k)_{k\in\N}$ and $(\overline{w}^k)_{k\in\N}$ be the sequences generated by Algorithm $\ref{alg:2}$ and $x\in C$. Then, for all $k\in\N$, it holds that
\begin{equation}
\label{eq:sta-lema2}
\begin{split}
\norm{x^{k+1}-x}^2 \leq & \norm{x^k-x}^2 + 2\alpha_k((f+g)(x) - (f+g)(x^k) + \epsilon_k)\\
& + \dfrac{\sigma^2}{(1-\sigma)^2}\alpha_k^2\norm{\overline{w}^k}^2 + \alpha_k^2\norm{u^k+w^k}^2,
\end{split}
\end{equation}
for any $w^k\in\partial g(x^k)$.
\end{lemma}

\begin{proof}
We first observe that the sequences calculated by Algorithm \ref{alg:2} satisfy equations \eqref{eq:eq0} and \eqref{eq:1} of Lemma \ref{lem:2}. Therefore, for any $x\in C$ and $k\in\N$, we have
\begin{equation}
\label{eq:4}
\begin{split}
\norm{x^{k+1} - x}^2 \leq & \norm{x^k-x}^2 - 2\alpha_k\inpr{u^k}{x^k-x} - 2\alpha_k\inpr{\overline{w}^k}{\overline{x}^k-x}\\
&  + \norm{\alpha_k\overline{w}^k + \overline{x}^k-y^k}^2
- \norm{y^k-\overline{x}^k}^2 + \norm{\alpha_ku^k}^2.
\end{split}
\end{equation}

From hypotheses $u^k\in\partial_{\epsilon_k} f(x^k)$ and $\overline{w}^k\in\partial_{\overline{\epsilon}_k}g(\overline{x}^k)$, it follows that
\begin{equation}
\label{eq:eps-sub-f-g}
\inpr{u^k}{x-x^k} \leq f(x) - f(x^k) + \epsilon_k, \quad\qquad \inpr{\overline{w}^k}{x-\overline{x}^k} \leq g(x) - g(\overline{x}^k) + \overline{\epsilon}_k.
\end{equation}
Hence, combining \eqref{eq:eps-sub-f-g} with  \eqref{eq:4} we have
\begin{equation*}
\begin{split}
\norm{x^{k+1}-x}^2 \leq & \norm{x^k-x}^2 + 2\alpha_k(f(x) - f(x^k) + \epsilon_k) + 2\alpha_k(g(x) - g(\overline{x}^k) + \overline{\epsilon}_k)\\
& +\norm{\alpha_k\overline{w}^k + \overline{x}^k-y^k}^2 - \norm{y^k-\overline{x}^k}^2 + \norm{\alpha_ku^k}^2.
\end{split}
\end{equation*}

Now, we use the error criterion of step 2 in Algorithm \ref{alg:2} to obtain
\begin{equation*}
\begin{split}
\norm{x^{k+1}-x}^2 \leq & \norm{x^k-x}^2 + 2\alpha_k(f(x) - f(x^k) + \epsilon_k) + 2\alpha_k(g(x) - g(\overline{x}^k)) + \sigma^2\norm{y^k-\overline{x}^k}\\
& - \norm{y^k-\overline{x}^k}^2 + \norm{\alpha_ku^k}^2.
\end{split}
\end{equation*}
Substituting \eqref{eq:3} into the equation above, and adding and subtracting $g(\overline{x}^k)$, we have
\begin{equation*}
\begin{split}
\norm{x^{k+1}-x}^2 \leq & \norm{x^k-x}^2 + 2\alpha_k((f+g)(x) - (f+g)(x^k) + \epsilon_k) + 2\alpha_k(g(x^k)-g(\overline{x}^k))\\
& + \sigma^2\norm{y^k-\overline{x}^k}^2 +2\alpha_k\inpr{u^k}{x^k-\overline{x}^k} - \norm{x^k-\overline{x}^k}^2.
\end{split}
\end{equation*}
Next, we take a subgradient $w^k\in\partial g(x^k)$, which exists since $x^k\in C\subseteq \text{int}(\dom{g})$, to obtain
\begin{equation*}
\begin{split}
\norm{x^{k+1}-x}^2 \leq & \norm{x^k-x}^2 + 2\alpha_k((f+g)(x) - (f+g)(x^k)+\epsilon_k) + 2\alpha_k\inpr{w^k}{x^k-\overline{x}^k}\\
& + \sigma^2\norm{y^k-\overline{x}^k}^2 + 2\alpha_k\inpr{u^k}{x^k-\overline{x}^k} - \norm{x^k-\overline{x}^k}^2.
\end{split}
\end{equation*}
Inequality above clearly implies
\begin{equation*}
\begin{split}
\norm{x^{k+1}-x}^2
\leq & \norm{x^k-x}^2 + 2\alpha_k((f+g)(x) - (f+g)(x^k)+\epsilon_k) + \sigma^2\norm{y^k-\overline{x}^k}^2\\
& + \alpha_k^2\norm{u^k+w^k}^2 - \norm{x^k-\overline{x}^k-\alpha_k(w^k+u^k)}^2.
\end{split}
\end{equation*}
Finally, we observe that the error criterion in step 2 of Algorithm \ref{alg:2} yields
\begin{equation*}
\norm{y^k-\overline{x}^k} \leq \dfrac{1}{1-\sigma}\norm{\alpha_k\overline{w}^k},
\end{equation*}
and combining the two relations above we obtain \eqref{eq:sta-lema2}.
\end{proof}

Associated with the sequence $(x^k)_{k\in\N}$ computed by Algorithm \ref{alg:2}, we can also define the sequence of functional values $((f+g)^k_\text{best})_{k\in\N}$ by equation \eqref{eq:best-func}, for keeping track of the current best functional value. 
Furthermore, we can obtain convergence rate results for Algorithm \ref{alg:2} similar to those established in Lemma \ref{lem:conv-rate-func-alg1}.

\begin{lemma}
\label{lem:conv-rate-func}
Let $(x^k)_{k\in\N}$ be the sequence generated by Algorithm $\ref{alg:2}$ and $((f+g)^k_\text{best})_{k\in\N}$ be its associated sequence of best functional values defined as in \eqref{eq:best-func}. If $S^\star\neq\emptyset$, then for all integer $k\geq0$, it holds that
\begin{equation}
\label{eq:conv-rate}
(f+g)^k_\text{best} - s^\star \leq \dfrac{d_0^2 + 2\sum_{j=0}^k\alpha_j\epsilon_j + \tilde{c}_k\sum_{j=0}^k\alpha_j^2}{2\sum_{j=0}^k\alpha_j},
\end{equation}
where $d_0$ is the distance of $x^0$ to $S^\star$, $\tilde{c}_k:=\max\limits_{j=0,\dots,k}\left\lbrace\dfrac{\sigma^2}{(1-\sigma)^2}\norm{\overline{w}^j}^2+\norm{u^j+w^j}^2\right\rbrace$ and $w^j\in\partial g(x^j)$.
\end{lemma}

\begin{proof}
By a similar argument to that used in Lemma \ref{lem:conv-rate-func-alg1}, we have that Lemma \ref{lem:1}, with $x=x^\star:=P_{S^\star}(x^0)$, implies
\begin{equation*}
\begin{split}
\norm{x^{k+1}-x^\star}^2
\leq & \norm{x^0-x^\star}^2 - 2\sum_{j=0}^k\alpha_j((f+g)(x^j)-s^\star) + 2\sum_{j=0}^k\alpha_j\epsilon_j\\
& + \sum_{j=0}^k\alpha_j^2\left(\dfrac{\sigma^2}{(1-\sigma)^2}\norm{\overline{w}^j}^2+\norm{u^j+w^j}^2\right).
\end{split}
\end{equation*}
Manipulating the above equation and using the definitions of $(f+g)^k_\text{best}$ and $\tilde{c}_k$, we deduce inequality \eqref{eq:conv-rate}.
\end{proof}

\section{Convergence analysis}\label{converg}
\label{sec:conv-ana}

This section is devoted to the convergence analysis of Algorithms \ref{alg:1} and \ref{alg:2}. 
The convergence results obtained below will depend on the stepsize rule that is chosen for the methods. 
The stepsize selection schemes we use here are very similar to the rules usually employed for the subgradient method. 
Specifically, we will consider three different strategies for choosing the stepsize sequence $(\alpha_k)_{k\in\N}$: \textit{(a)} constant stepsize, \textit{(b)} diminishing stepsize, \textit{(c)} a stepsize based on a
subgradient step length choice due to Polyak.
With these rules we will obtain various types of convergence results for Algorithms \ref{alg:1} and \ref{alg:2}. For instance, it will be proven convergence in the objective values and
convergence to a neighborhood  of the optimal set.
The convergence analyses of the two methods share underlying elements and we present them in a unified way.

We remark that, as in subgradient methods, the convergence properties of Algorithms \ref{alg:1} and \ref{alg:2} will hinge on monitoring the distance of the iterates to the optimal set, which can be achieved through Lemmas \ref{lem:2} and \ref{lem:1}. 
Indeed, if for example we assume that there exists an optimal solution $x^\star\in S^\star$, then Lemma \ref{lem:2} with $x=x^\star$ implies that the sequence $(x^k)_{k\in\N}$, calculated by Algorithm \ref{alg:1}, satisfies
\begin{equation*}
\norm{x^{k+1}-x^\star}^2 < \norm{x^k-x^\star}^2 + r_k^2,
\end{equation*}
provided that the stepsize $\alpha_k$ is such that
\begin{equation}
\label{eq:stepsize1}
0 < \alpha_k < 2\dfrac{(f+g)(x^k)-s^\star-\epsilon_k}{\norm{u^k+w^k}^2}.
\end{equation}
Similarly, if $(x^k)_{k\in\N}$ is the sequence generated by Algorithm \ref{alg:2}, then by Lemma \ref{lem:1} we have
\begin{equation*}
\norm{x^{k+1}-x^\star} < \norm{x^k-x^\star}
\end{equation*}
for all stepsizes $\alpha_k$ such that
\begin{equation}
\label{eq:stepsize}
0 < \alpha_k < 2\dfrac{(f+g)(x^k)-s^\star-\epsilon_k}{\dfrac{\sigma^2}{(1-\sigma)^2}\norm{\overline{w}^k}^2+\norm{u^k+w^k}^2}.
\end{equation}

Hence, in the case of choosing $\alpha_k$ such that it belongs to the interval determined by \eqref{eq:stepsize1} (resp. \eqref{eq:stepsize}) we can guarantee that the sequence $(x^k)_{k\in\N}$, calculated by Algorithm \ref{alg:1} (resp. Algorithm \ref{alg:2}), is quasi-Fejér (resp. Fejér) convergent to $S^\star$. Furthermore, if for Algorithm \ref{alg:2} we select $\alpha_k$ such that it satisfies the inequalities in \eqref{eq:stepsize}, then the distance of the current iterate to $x^\star$ is reduced. 

However, we observe that these rules for selecting the stepsize have as a practical disadvantage that they require prior knowledge of the optimal value $s^\star$, which is usually not available.
Moreover, if we want to use rule \eqref{eq:stepsize} for Algorithm \ref{alg:2}, we also need to know, at the beginning of iteration $k$, the  $\overline{\epsilon}_k$-subgradient $\overline{w}^k$ of $g$ at $\overline{x}^k$, which is only calculated in step 2 of the method, and this calculation obviously depends on the stepsize $\alpha_k$. Thus, in the case of Algorithm \ref{alg:2}, we cannot use this scheme to choose $\alpha_k$.

One alternative, assuming that there is a constant $c>0$ such that $\norm{\overline{w}^k}\leq c$ for all $k\in\N$, is to select $\alpha_k$ such that
\begin{equation}
\label{eq:stepsize-1}
0 < \alpha_k < 2\dfrac{(f+g)(x^k)-s^\star - \epsilon_k}{\dfrac{\sigma^2}{(1-\sigma)^2}c^2+ \norm{u^k+w^k}^2}.
\end{equation}
Note that by Lemma \ref{lem:1}, choosing $\alpha_k$ by \eqref{eq:stepsize-1} 
also guarantees that the distance of the current iterate to the optimal solution set is reduced.

In our convergence analysis of Algorithms \ref{alg:1} and \ref{alg:2}, we will assume the following condition:
\begin{enumerate}
\item[(A.3)] there exists $c>0$ such that $\max\{\norm{\overline{w}^k},\norm{w^k},\norm{u^k}\}\leq c$ for every $k\in\N$ and some $w^k\in\partial g(x^k)$.
\end{enumerate}

Assumption A.3 has been used in several recent works. As well as we know, the first time this assumption appears was in \cite{nedic}. Others works that had used it are \cite{BelloDiaz,Bert2011, phdthesis}. Note that when $\HH$ is finite-dimensional and $C\subseteq int(\dom f\cap \dom g)$, assumption A.3 is automatically satisfied. Also, when the functions $f$ and $g$ are polyhedral (i.e., $f$ and $g$ are the pointwise maximum of a finite number of affine functions) the assumption is satisfied.

\subsection{Constant stepsize}
\label{sub:constant-step}

In this subsection, we study the convergence of Algorithms \ref{alg:1} and \ref{alg:2} when the sequence of stepsizes $(\alpha_k)_{k\in\N}$ is fixed at some positive scalar $\alpha$.
For this case, it can be ensured that
$(x^k)_{k\in\N}$ can get to within an $\mathcal{O}(\alpha)$-neighborhood  of the $\epsilon$\emph{-optimal set}, where $\epsilon=\lim_{k\to\infty}\epsilon_k$.
The proof presented below follows the lines of the proof of convergence of the subgradient methods presented in \cite{bertsekas2014convex}.

\begin{theorem}
\label{theo:const-step}
Let $(x^k)_{k\in\N}$ be the sequence generated by any one of the Algorithms $\ref{alg:1}$-$\ref{alg:2}$ and $c>0$ be given in \emph{A.3}.
Assume that $\alpha_k\equiv\alpha$, where $\alpha>0$, and $\lim_{k\to\infty}\epsilon_k=\epsilon$.
Then, the following statements hold.
\begin{enumerate}
\item[(i)] If $s^\star=-\infty$, then\, $\liminf_{k\to\infty}(f+g)(x^k)=s^\star$.
\item[(ii)] If $s^\star>-\infty$ and $(x^k)_{k\in\N}$ is the sequence calculated by Algorithm $\ref{alg:1}$, then
$$\liminf_{k\to\infty}(f+g)(x^k) \leq s^\star  + \epsilon + 2\alpha c^2.$$
\item[(iii)] If $s^\star>-\infty$ and $(x^k)_{k\in\N}$ is the sequence calculated by Algorithm $\ref{alg:2}$, then
$$\liminf_{k\to\infty}(f+g)(x^k)\leq s^\star+\epsilon+\dfrac{\alpha C_\sigma c^2}{2},$$
where\, $C_\sigma=\dfrac{\sigma^2}{(1-\sigma)^2}+4$.
\end{enumerate}
\end{theorem}

\begin{proof}
We first prove (i) and (ii) simultaneously.
We assume that $(x^k)_{k\in\N}$ is the sequence generated by Algorithm \ref{alg:1} and denote $l_{\infty}=\lim\inf_{k\to\infty}(f+g)(x^k)$. Now, we suppose by contradiction that there exists $\eta>0$ such that
 \begin{equation*}
 l_{\infty} > s^\star + \epsilon + 2\alpha c^2 + 2\eta.
 \end{equation*}
Since $s^\star=\inf_{x\in C}(f+g)(x)$, from equation above it follows that there is $\hat{x}\in C$ such that
\begin{equation*}
l_{\infty} \geq (f+g)(\hat{x}) + \epsilon + 2\alpha c^2 +  2\eta,
\end{equation*}
 and by the definition of $l_\infty$ we have that there exists $k_0$ large enough satisfying
 \begin{equation*}
 (f+g)(x^k) \geq l_\infty - \eta,\qquad\quad\text{for all }\, k\geq k_0.
 \end{equation*}
Adding the two inequalities above we obtain
 \begin{equation*}
 (f+g)(x^k) - (f+g)(\hat{x}) \geq \epsilon + 2\alpha c^2 + \eta.
 \end{equation*}
Now, we apply Lemma \ref{lem:2} with $x=\hat{x}$ and combine with the equation above to obtain
 \begin{equation}
 \label{eq:lem-1}
 \begin{split}
 \norm{x^{k+1}-\hat{x}}^2 \leq & \norm{x^k-\hat{x}}^2 - 2\alpha(\epsilon + 2\alpha c^2 + \eta - \epsilon_k) + r_k^2 + \alpha^2\norm{u^k+w^k}^2\\
 \leq & \norm{x^k-\hat{x}}^2 - 2\alpha(\epsilon + 2\alpha c^2 + \eta - \epsilon_k) + r_k^2 + \alpha^24c^2\\
 = & \norm{x^k-\hat{x}}^2 - 2\alpha(\eta+\epsilon-\epsilon_k) + r_k^2,
 \end{split}
 \end{equation}
 where the second inequality above follows from assumption A.3.

Since $\epsilon_k\to\epsilon$, we can assume that $k_0$ is large enough such that
 \begin{equation}
 \label{eq:n-1}
 \eta+\epsilon - \epsilon_k \geq \dfrac{\eta}{2}\qquad\qquad\text{for all }\,k\geq k_0.
 \end{equation}
 Therefore, substituting equation above into \eqref{eq:lem-1} we have
 \begin{equation}
 \label{eq:lem-2}
 \begin{split}
 \norm{x^{k+1}-\hat{x}}^2 \leq \norm{x^k-\hat{x}}^2 - \alpha\eta + r_k^2 \leq \dots \leq\norm{x^{k_0}-\hat{x}}^2 - (k+1-k_0)\eta\alpha + \sum_{j=k_0}^kr_j^2
 \end{split}
 \end{equation}
 for all integer $k\geq k_0$. Since $(r_k)_{k\in\N}\in\ell_2(\R)$, it follows that equation above cannot hold for $k$ sufficiently large, hence we arise to a contradiction.

We now assume that $(x^k)_{k\in\N}$ is calculated by Algorithm \ref{alg:2} and prove item (i) and (iii) for this case. Since the proof is very similar to the previous one, we will only indicate the necessary modifications.

We define $l_\infty:=\lim\inf_{k\to\infty}(f+g)(x^k)$ and suppose by contradiction that there exists $\eta>0$ such that
  \begin{equation*}
 l_{\infty} > s^\star + \epsilon + \dfrac{\alpha C_\sigma c^2}{2} + 2\eta.
 \end{equation*}
 Therefore, using the definition of $l_\infty$, Lemma \ref{lem:1} and hypothesis A.3, in place of equation \eqref{eq:lem-1}, we have that there exist $\hat{x}\in C$ and $k_0\in\N$ such that
 \begin{equation*}
 \begin{split}
 \norm{x^{k+1}-\hat{x}}^2 \leq & \norm{x^k-\hat{x}}^2 - 2\alpha\left(\epsilon + \dfrac{\alpha C_\sigma c^2}{2} + \eta - \epsilon_k\right) + \alpha^2c^2C_\sigma\\
 = & \norm{x^k-\hat{x}}^2 - 2\alpha(\eta+\epsilon-\epsilon_k),
 \end{split}
 \end{equation*}
 for all $k\geq k_0$. Since $\epsilon_k\to\epsilon$, as before, we can take $k_0$ large enough such that \eqref{eq:n-1} holds.
 Combining these equations, in analogy with \eqref{eq:lem-2}, we obtain
  \begin{equation*}
 \begin{split}
 \norm{x^{k+1}-\hat{x}}^2 \leq \norm{x^k-\hat{x}}^2 - \alpha\eta\leq\dots\leq\norm{x^{k_0}-\hat{x}}^2 - (k+1-k_0)\eta\alpha,
 \end{split}
 \end{equation*}
for all integer $k\geq k_0$, which cannot hold for $k$ sufficiently large. Thus, we have a contradiction.
\end{proof}

Next result establishes complexity bounds for Algorithms \ref{alg:1} and \ref{alg:2} to reach a given level of optimality up to the threshold tolerance of the preceding theorem. We note that the complexity estimates for Algorithm \ref{alg:1} will depend on the sequence $(r_k)_{k\in\N}$.

\begin{lemma}
\label{lem:compl-const}
Assume that $S^\star\neq\emptyset$ and $\epsilon_k\equiv\epsilon$, where $\epsilon\geq0$. Define $d_0:=\dist{x^0}{S^\star}$ and $r:=\norm{(r_k)_{k\in\N}}_2$. Then, the following statements hold.
\begin{enumerate}
\item[(i)] If $(x^k)_{k\in\N}$ is the sequence generated by Algorithm $\ref{alg:1}$, then for all $\delta>0$ there exists an index $i=\mathcal{O}\left(\dfrac{d_0^2+r^2}{\delta\alpha}\right)$ such that
\begin{equation*}
(f+g)(x^i) \leq s^\star + \epsilon + 2c^2\alpha + \delta.
\end{equation*}

\item[(ii)] If $(x^k)_{k\in\N}$ is the sequence calculated by Algorithm $\ref{alg:2}$, then for all $\delta>0$ there is an index $i=\mathcal{O}\left(\dfrac{d_0^2}{\delta\alpha}\right)$ such that
\begin{equation*}
(f+g)(x^i) \leq s^\star + \epsilon + \dfrac{c^2C_\sigma\alpha}{2} + \delta.
\end{equation*}
\end{enumerate}
\end{lemma}

\begin{proof}
\item[(i)] We observe that, since $\alpha_k\equiv\alpha$, Lemma \ref{lem:conv-rate-func-alg1} implies
\begin{equation*}
(f+g)^k_\text{best} - s^\star \leq \dfrac{d_0^2}{2(k+1)\alpha} + \epsilon + \dfrac{r^2}{2(k+1)\alpha} + \dfrac{4c^2\alpha}{2},
\end{equation*}
where it was used in the inequality above that $c_k\leq4c^2$ for all $k\in\N$. The assertion follows as a direct consequence of equation above and the definition of the sequence of the best functional values.

\item[(ii)] Since $\tilde{c}_k\leq c^2C_\sigma$ and $\alpha_k=\alpha$ for all $k\in\N$, Lemma \ref{lem:conv-rate-func} yields
\begin{equation*}
(f+g)^k_\text{best} - s^\star \leq \dfrac{d_0^2}{2(k+1)\alpha} + \epsilon + \dfrac{c^2C_\sigma\alpha}{2}.
\end{equation*}
Thus, the statement follows immediately from the definition of the sequence $((f+g)^k_\text{best})_{k\in\N}$ and the above inequality.
\end{proof}

\subsection{Diminishing stepsize}

In this subsection, we analyze the convergence of the inexact P$\epsilon$-SMs presented in sections \ref{asbt} and \ref{relat}, when the sequence of stepsizes $(\alpha_k)_{k\in\N}$ satisfies $\lim_{k\to\infty}\alpha_k=0$ and $\sum_{k=0}^\infty\alpha_k=\infty$. In this case, we are able to establish (under suitable assumptions) exact convergence for the sequences calculated by Algorithms \ref{alg:1} and \ref{alg:2}, provided that there exists a solution of problem \eqref{eq:problem}. Otherwise, we will have that $(x^k)_{k\in\N}$ is an unbounded sequence.

\begin{theorem}
\label{theo:fejer}
Let $(x^k)_{k\in\N}$ be the sequence generated by any one of the Algorithms $\ref{alg:1}$-$\ref{alg:2}$. Suppose that the sequences $(\alpha_k)_{k\in\N}$ and $(\epsilon_k)_{k\in\N}$ are such that $\lim_{k\to\infty}\alpha_k=0$, $\sum\limits_{k=0}^\infty\alpha_k=\infty$ and $\lim\limits_{k\to\infty}\epsilon_k=\epsilon$.
Then, we have
\begin{itemize}
\item[(i)] $\lim\inf_{k\to\infty}(f+g)(x^k)\leq s^\star+\epsilon$.
\end{itemize}

\noindent In addition, assume that $S^\star\neq \emptyset$,\, $\sum\limits_{k=0}^\infty \alpha_k^2<\infty$ and $\sum\limits_{k=0}^\infty\alpha_k\epsilon_k<\infty$, then

\begin{itemize}
\item[(ii)] $(x^k)_{k\in\N}$ is quasi-Fej\'er convergent to the solution set $S^\star$;
\item[(iii)] $(x^k)_{k\in\N}$ is bounded;
\item[(iv)] $\epsilon=0$, $\lim_{k\to\infty}(f+g)(x^k)=s^\star$ and the sequence $(x^k)_{k\in\N}$ is weakly convergent to a point $x^\star\in S^\star$.
\end{itemize}
\end{theorem}

\begin{proof}
   \item[(i)]
It can be proven with analysis similar to that in the proof of Theorem \ref{theo:const-step}, and for the sake of brevity we omit its proof.

    \item[(ii)] Follows by assumption A.3 and taking $x=x^\star\in S^\star$ in Lemmas \ref{lem:2} and \ref{lem:1}.

    \item[(iii)] It is a direct consequence of the previous item (see Proposition \ref{prop:fejer}(i)).

 \item[(iv)] Since we are assuming that $\sum\limits_{k=0}^\infty\alpha_k\epsilon_k<\infty$, $\sum\limits_{k=0}^\infty\alpha_k=\infty$ and $\lim_{k\to\infty}\epsilon_k=\epsilon$, it is clear that we necessarily have $\epsilon=0$. Hence, from item (i) it follows that $\liminf_{k\to\infty} (f+g)(x^k)\leq s^\star$, which implies $\liminf_{k\to\infty} (f+g)(x^k)=s^\star$.

 Now, we assume that $(x^k)_{k\in\N}$ is the sequence computed by Algorithm \ref{alg:1}, and define $\beta_k = (f+g)(x^k) - s^\star\geq0$ for all $k\in\N$. We take $x^\star\in S^\star$ and apply Lemma \ref{lem:2} with $x=x^\star$ to have
 \begin{align*}
 \norm{x^{k+1}-x^\star}^2 \leq &\norm{x^k-x^\star}^2 -2\alpha_k\beta_k + 2\alpha_k\epsilon_k + r_k^2 + 4\alpha_k^2c^2\\
 \vdots & \\
  \leq & \norm{x^0-x^\star}^2 - 2\sum_{j=0}^k\alpha_j\beta_j + 2\sum_{j=0}^k\alpha_j\epsilon_j + \sum_{j=0}^kr_j^2 + 4c^2\sum_{j=0}^k\alpha_j^2.
 \end{align*}
Rearranging the terms in the inequality above we obtain 
  \begin{equation*}
 2\sum_{j=0}^\infty \alpha_j\beta_j \leq \norm{x^0-x^\star} ^2 + 2\sum_{j=0}^\infty \alpha_j\epsilon_j + \sum_{j=0}^\infty r_j^2 + 4c^2\sum_{j=0}^\infty \alpha_j^2 <\infty.
 \end{equation*}
Hence, we apply \cite[Proposition 2]{alber1998projected} to the sequences $(\beta_k)_{k\in\N}$ and $(\alpha_k)_{k\in\N}$ to conclude that $\lim_{k\to\infty}\beta_k=0$, which directly implies $\lim_{k\to\infty}(f+g)(x^k)=s^\star$

 For Algorithm \ref{alg:2}, we define the sequence $(\beta_k)_{k\in\N}$ as before and apply Lemma \ref{lem:1} to obtain
 \begin{equation*}
 \begin{split}
 \norm{x^{k+1}-x^\star}^2 \leq \norm{x^0-x^\star}^2 - 2\sum_{j=0}^k\alpha_j\beta_j + 2\sum_{j=0}^k\alpha_j\epsilon_j + c^2C_\sigma\sum_{j=0}^k\alpha_j^2.
 \end{split}
 \end{equation*}
Therefore, from inequality above we have
 \begin{equation*}
 2\sum_{j=0}^\infty \alpha_j\beta_j \leq \norm{x^0-x^\star} ^2 + 2\sum_{j=0}^\infty\alpha_j\epsilon_j + c^2C_\sigma\sum_{j=0}^\infty\alpha_j^2 <\infty,
 \end{equation*}
and applying \cite[Proposition 2]{alber1998projected} to the sequences $(\beta_k)_{k\in\N}$ and $(\alpha_k)_{k\in\N}$ we deduce that $\lim_{k\to\infty}\beta_k=0$, which yields $\lim_{k\to\infty}(f+g)(x^k)=s^\star$.

Thus, we have that the sequence $(x^k)_{k\in\N}$ calculated by any one of the Algorithms \ref{alg:1}-\ref{alg:2} satisfies that $\lim_{k\to\infty}(f+g)(x^k)=s^\star$.
To conclude the proof we take $\tilde{x}$ a week accumulation point of $(x^k)_{k\in\N}$, which exists by item (ii), and $(x^{j_k})_{k\in\N}$ a subsequence such that $x^{j_k}\rightharpoonup\tilde{x}$ as $k\to\infty$. Since $f+g$ is a closed function, it follows that
\begin{equation*}
(f+g)(\tilde{x}) \leq \liminf_{k\to\infty}(f+g)(x^{j_k})=\lim_{k\to\infty}(f+g)(x^k)=s^\star.
\end{equation*}
Hence, we deduce that $\tilde{x}\in S^\star$
and because $\tilde{x}$ was arbitrary we have that all weak accumulation points of $(x^k)_{k\in\N}$ belong to $S^\star$. Due to item (ii) and Proposition \ref{prop:fejer}(ii), we conclude that the sequence $(x^k)_{k\in\N}$ converges weakly to some point in the solution set $S^\star$.
\end{proof}

\begin{theorem}
Assume the hypotheses of Theorem $\ref{theo:fejer}$. In addition, suppose that $\epsilon=0$ and $S^\star=\emptyset$. Then, the sequence $(x^k)_{k\in\N}$ generated by any one of the Algorithms $\ref{alg:1}$-$\ref{alg:2}$ is unbounded.
\end{theorem}

\begin{proof}
Suppose by contradiction that $(x^k)_{k\in\N}$ is a bounded sequence and take a subsequence $(x^{j_k})_{k\in\N}$ such that $\lim_{k\to\infty}(f+g)(x^{j_k})=\liminf_{k\to\infty}(f+g)(x^k)$. Refining $(x^{j_k})_{k\in\N}$ if necessary, we may assume that it converges weakly to some point $\tilde{x}\in C$. Since $f+g$ is closed, it follows that
\begin{equation*}
(f+g)(\tilde{x}) \leq \liminf_{k\to\infty}(f+g)(x^{j_k})= \lim_{k\to\infty}(f+g)(x^{j_k}) = \liminf_{k\to\infty}(f+g)(x^k)\leq s^\star,
\end{equation*}
where the last inequality above follows from item (i) in Theorem \ref{theo:fejer} and the assumption that $\epsilon=0$. Equation above clearly implies that $\tilde{x}\in S^\star$, which is a contradiction.
\end{proof}

\subsection{Polyak stepsize}

In \cite{Pol87} Polyak suggested a stepsize for subgradient methods that can be used when the optimal value of 
the problem is known. Since in practical problems this data is usually not available, 
other modifications of Polyak rule have been proposed that replace the unknown optimal value with an estimate, see for instance \cite{bertsekas2014convex}.
In this subsection, we present stepsize rules for Algorithms \ref{alg:1} and \ref{alg:2} that are inspired in Polyak's step rule, and we analyze the convergence of these methods. We consider both cases, when $s^\star$ is known and replacing it with an estimate.

We observe that the rules considered in this subsection can also be motivated by equations \eqref{eq:stepsize1} and \eqref{eq:stepsize-1}. Indeed, if we replace the estimate $s_k$ by $s^\star$ in equations \eqref{eq:pol-step} and \eqref{eq:pol-step1} below, then these choices of $\alpha_k$ satisfy the desired inequalities in \eqref{eq:stepsize} and \eqref{eq:stepsize-1}, which guarantees the quasi-Fejér convergence of $(x^k)_{k\in\N}$ to $S^\star$.

If $x^k$ is the iterated calculated with Algorithm \ref{alg:1} at iteration $k$, we fix $w^k\in\partial g(x^k)$ such that it satisfies condition A.3, and then we define the stepsize $\alpha_k$ as
\begin{equation}
\label{eq:pol-step}
\alpha_k : =\gamma_k\dfrac{(f+g)(x^k)-s_k-\epsilon_k}{\norm{u^k+w^k}^2}.
\end{equation}
Similarly, if $x^k$ is the current iterated generated by Algorithm \ref{alg:2}, then we choose $w^k\in\partial g(x^k)$ satisfying A.3 and define the stepsize $\alpha_k$ by
\begin{equation}
\label{eq:pol-step1}
\alpha_k := \gamma_k\dfrac{(f+g)(x^k)-s_k-\epsilon_k}{\sigma^2c^2/(1-\sigma)^2+\norm{u^k+w^k}^2}.
\end{equation}

In equations \eqref{eq:pol-step} and \eqref{eq:pol-step1} we are assuming that $0<\underline{\gamma}\leq\gamma_k\leq\overline{\gamma}<2$ for all $k\in\N$, and $(s_k)_{k\in\N}$ is a monotone decreasing sequence that converges to some $\tilde{s}$ and such that $s_k+\epsilon_k<(f+g)(x_k)$ for every $k\in\N$.

To analyze the convergence of Algorithms \ref{alg:1} and \ref{alg:2} with $\alpha_k$ given by \eqref{eq:pol-step} and \eqref{eq:pol-step1}, respectively, we will first consider the case where $\tilde{s}\geq s^\star$.

\begin{theorem}
\label{theo:1}
Let $(x^k)_{k\in\N}$ be the sequence generated by Algorithm $\ref{alg:1}$ with $\alpha_k$ as in \eqref{eq:pol-step} and assume that $\lim_{k\to\infty}\epsilon_k=\epsilon$.
Define the level set
\begin{equation}
\label{eq:lev-set}
L_{f+g}(\tilde{s}) := \set{x\in C\,:\,f(x)+g(x)\leq\tilde{s}}
\end{equation}
and suppose that $L_{f+g}(\tilde{s})\neq\emptyset$. Then, the following statements hold.
\begin{itemize}
\item[(i)] $(x^k)_{k\in\N}$ is quasi-Fejér convergent to $L_{f+g}(\tilde{s})$;
\item[(ii)] $\lim_{k\to\infty}(f+g)(x^k)= \tilde{s}+\epsilon$.
\end{itemize}

\noindent In addition, if $\epsilon=0$, then

\begin{itemize}
\item[(iii)] $(x^k)_{k\in\N}$ is weakly convergent to some point in $L_{f+g}(\tilde{s})$.
\end{itemize}
\end{theorem}

\begin{proof}
\item[(i)] We take $\tilde{x}\in L_{f+g}(\tilde{s})$ and apply Lemma \ref{lem:2} with $x=\tilde{x}$ to obtain
\begin{equation*}
\begin{split}
\norm{x^{k+1}-\tilde{x}}^2 \leq & \norm{x^k-\tilde{x}}^2 + 2\alpha_k((f+g)(\tilde{x}) - (f+g)(x^k)+\epsilon_k) + r_k^2 + \alpha_k^2\norm{u^k+w^k}^2\\
\leq & \norm{x^k-\tilde{x}}^2
+ 2\alpha_k(s_k - (f+g)(x^k)+\epsilon_k) + r_k^2 + \alpha_k^2\norm{u^k+w^k}^2\\
= & \norm{x^k-\tilde{x}}^2 - \gamma_k(2-\gamma_k)\dfrac{((f+g)(x^k)-s_k-\epsilon_k)^2}{\norm{u^k+w^k}^2}+ r_k^2,
\end{split}
\end{equation*}
where the second inequality above follows from the assumption $(f+g)(\tilde{x})\leq\tilde{s}\leq s_k$, and the equality is due to the definition of $\alpha_k$ in \eqref{eq:pol-step}. Since $\gamma_k\in[\underline{\gamma},\overline{\gamma}]$, we have
\begin{equation}
\label{eq:eq-3}
\norm{x^{k+1}-\tilde{x}}^2 \leq \norm{x^{k}-\tilde{x}}^2 - \underline{\gamma}(2-\overline{\gamma})\dfrac{((f+g)(x^k)-s_k-\epsilon_k)^2}{\norm{u^k+w^k}^2}+r_k^2.
\end{equation}
Thus, item (i) follows directly from equation above and Definition \ref{def:quasi-fejer}.

\item[(ii)] We combine \eqref{eq:eq-3} with the fact that $\underline{\gamma},\,\overline{\gamma}\in\left]0,2\right[$ and $\norm{u^k+w^k}^2\leq4c^2$ for all $k\in\N$, to obtain
\begin{equation*}
\norm{x^{k+1}-\tilde{x}}^2 \leq \norm{x^k-\tilde{x}}^2-\dfrac{\underline{\gamma}(2-\overline{\gamma})}{4c^2}((f+g)(x^k)-s_k-\epsilon_k)^2 + r_k^2.
\end{equation*}
Applying recursively the equation above and rearranging the terms, we have
\begin{equation}
\label{eq:5}
\dfrac{\underline{\gamma}(2-\overline{\gamma})}{4c^2}\sum_{j=0}^k((f+g)(x^j) - s_j -\epsilon_j)^2 \leq \norm{x^0-\tilde{x}}^2 + \sum_{j=0}^kr_j^2.
\end{equation}
Since inequality above holds for all $k\in\N$ and
$(r_k)_{k\in\N}\in\ell_2(\R)$, we conclude that
\begin{equation*}
\sum_{j=0}^\infty ((f+g)(x^j) - s_j - \epsilon_j) ^2 < + \infty.
\end{equation*}
From which we deduce item (ii).
\item[(iii)] Assuming now that $\epsilon=0$, from item (ii) we have $\lim_{k\to\infty}(f+g)(x^k)=\tilde{s}$. Moreover, item (i), together with Proposition \ref{prop:fejer}(i), yields that $(x^k)_{k\in\N}$ is a bounded sequence. Therefore, taking $\tilde{x}$ a weak accumulation point of this sequence and $(x^{j_k})_{k\in\N}$ a subsequence such that $x^{j_k}\rightharpoonup\tilde{x}$ as $k\to\infty$, it follows that $\lim_{k\to\infty}(f+g)(x^{j_k})=\tilde{s}$. Since $f+g$ is a closed function, we have
\begin{equation*}
(f+g)(\tilde{x}) \leq \liminf_{k\to\infty}(f+g)(x^{j_k})=\tilde{s}.
\end{equation*}
Hence, $\tilde{x}\in L_{f+g}(\tilde{s})$ and because $\tilde{x}$ was an arbitrary accumulation point of $(x^k)_{k\in\N}$, we conclude the proof using Proposition \ref{prop:fejer}(ii).
\end{proof}

Similar results to those presented in Theorem \ref{theo:1} can be established for the sequence generated by Algorithm \ref{alg:2} with stepsize defined by \eqref{eq:pol-step1}, as it is shown in the following result. The theorem below may be proven in much the same way as Theorem \ref{theo:1}.

\begin{theorem}
\label{theo:2}
Let $(x^k)_{k\in\N}$ be the sequence generated by Algorithm $\ref{alg:2}$ with $\alpha_k$ as in \eqref{eq:pol-step1}. Suppose that the level set $L_{f+g}(\tilde{s})$ defined in \eqref{eq:lev-set} is non-empty and $\lim_{k\to\infty}\epsilon_k=\epsilon$. Then, the following statements hold.
\begin{itemize}
\item[(i)] $(x^k)_{k\in\N}$ is Fejér convergent to $L_{f+g}(\tilde{s})$;
\item[(ii)] $\lim_{k\to\infty}(f+g)(x^k)= \tilde{s}+\epsilon$.
\end{itemize}

\noindent In addition, if $\epsilon=0$, then

\begin{itemize}
\item[(iii)] $(x^k)_{k\in\N}$ is weakly convergent to some point in $L_{f+g}(\tilde{s})$.
\end{itemize}
\end{theorem}

\begin{proof}
Taking $\tilde{x}\in L_{f+g}(\tilde{s})$ and using Lemma \ref{lem:1} and \eqref{eq:pol-step1}, in place of \eqref{eq:eq-3}, we obtain
\begin{equation}
\label{eq:eq4}
\norm{x^{k+1}-\tilde{x}}^2 \leq \norm{x^k-\tilde{x}}^2 - \underline{\gamma}(2-\overline{\gamma})\dfrac{((f+g)(x^k)-s_k-\epsilon_k)^2}{\sigma^2c^2/(1-\sigma)^2+\norm{u^k+w^k}^2}.
\end{equation}
The rest of the proof runs as before.
\end{proof}

In the special case where $s^\star$ is known and finite, and we take $\epsilon_k\equiv0$, we can define the stepsize $\alpha_k$ by
\begin{equation}
\label{eq:step-sizes-pol}
\begin{split}
\alpha_k := \gamma_k\dfrac{(f+g)(x^k)-s^\star}{\norm{u^k+w^k}^2}\qquad\text { and }\qquad \alpha_k := \gamma_k\dfrac{(f+g)(x^k)-s^\star}{\sigma^2c^2/(1-\sigma)^2+\norm{u^k+w^k}^2}
\end{split}
\end{equation}
for Algorithms \ref{alg:1} and \ref{alg:2}, respectively. Hence, as a consequence of Theorems \ref{theo:1} and \ref{theo:2} we can prove convergence to a solution of problem \eqref{eq:problem}.

\begin{corollary}
\label{cor:1}
Let $(x^k)_{k\in\N}$ be the sequence generated by any one of the Algorithms $\ref{alg:1}$-$\ref{alg:2}$ with $\alpha_k$ given by \eqref{eq:step-sizes-pol}, and assume that $S^\star\neq\emptyset$. Then, the following statements hold.
\begin{enumerate}
\item[(i)] $(x^k)_{k\in\N}$ is quasi-Fejér convergent to $S^\star$ (if $(x^k)_{k\in\N}$ is calculated by Algorithm $\ref{alg:2}$, then the sequence is actually Fejér convergent to $S^\star$);
\item[(ii)] $\lim_{k\to\infty}(f+g)(x^k)=s^\star$;
\item[(iii)] $(x^k)_{k\in\N}$ is weakly convergent to some point in $S^\star$.
\end{enumerate}
\end{corollary}

\begin{proof}
The corollary is a direct consequence of Theorems \ref{theo:1} and \ref{theo:2} fixing $s_k\equiv s^\star$, $\tilde{s}=s^\star$ and $\epsilon_k\equiv0$.
\end{proof}

We now establish convergence rates for the functional values for Algorithms \ref{alg:1} and \ref{alg:2}, using Polyak stepsize rules \eqref{eq:pol-step} and \eqref{eq:pol-step1}. It is clear that the number of iterates generated by Algorithm \ref{alg:1} will depend on the $\ell_2$-norm of the sequence $(r_k)_{k\in\N}$.

\begin{theorem}
\label{theo:compl-pol}
Assume that $L_{f+g}(\tilde{s})\neq\emptyset$ and $\epsilon_k\equiv\epsilon$. Then, the following statements hold.
\begin{enumerate}
\item[(i)] Let $(x^k)_{k\in\N}$ be the sequence generated by Algorithm $\ref{alg:1}$ with stepsize given by \eqref{eq:pol-step}. Then, for every $k\in\N$, we have
\begin{equation*}
(f+g)_\text{best}^k-\tilde{s}\leq \dfrac{\dist{x^0}{L_{f+g}(\tilde{s})}2c}{\sqrt{\underline{\gamma}(2-\overline{\gamma})}\sqrt{k+1}} + \dfrac{r^2}{\sqrt{k+1}} + \epsilon,
\end{equation*}
where $r:=\norm{(r_k)_{k\in\N}}_2$.

\item[(ii)] Let $(x^k)_{k\in\N}$ be calculated by Algorithm $\ref{alg:2}$ with $\alpha_k$ defined by \eqref{eq:pol-step1}. Then, for all $k\in\N$, we have
\begin{equation*}
(f+g)_\text{best}^k - \tilde{s} \leq \dfrac{\dist{x^0}{L_{f+g}(\tilde{s})}c\sqrt{C_\sigma}}{\sqrt{\underline{\gamma}(2-\overline{\gamma})}\sqrt{k+1}} + \epsilon.
\end{equation*}
\end{enumerate}
\end{theorem}

\begin{proof}
\item[(i)] From equation \eqref{eq:5} with $\tilde{x}=P_{L_{f+g}(\tilde{s})}(x^0)$ it follows that
\begin{equation*}
\sum_{j=0}^k((f+g)(x^j)-\tilde{s}-\epsilon)^2\leq\sum_{j=0}^k((f+g)(x^j)-s_k-\epsilon)^2 \leq \dfrac{4c^2\dist{x^0}{L_{f+g}(\tilde{s})}^2}{\underline{\gamma}(2-\overline{\gamma})} + r^2,
\end{equation*}
where the first inequality above is due to the assumption that $s_k\geq\tilde{s}$ for all $k\in\N$. Using inequality above and the definition of $(f+g)_\text{best}^k$ we have
\begin{equation*}
(k+1)((f+g)_\text{best}^k-\tilde{s}-\epsilon)^2\leq \dfrac{4c^2\dist{x^0}{L_{f+g}(\tilde{s})}^2}{\underline{\gamma}(2-\overline{\gamma})} + r^2,
\end{equation*}
from which follows the desired estimate.

\item[(ii)] Analysis similar to that in the proof of item (i) holds for this case. Indeed, applying recursively equation \eqref{eq:eq4} with $\tilde{x}=P_{L_{f+g}(\tilde{s})}(x^0)$ and rearranging the terms, we obtain
\begin{equation*}
\underline{\gamma}(2-\overline{\gamma})\sum_{j=0}^k\dfrac{((f+g)(x^j)-s_j-\epsilon)^2}{\sigma^2c^2/(1-\sigma)^2+\norm{u^j+w^j}^2} \leq \dist{x^0}{L_{f+g}(\tilde{x})}^2.
\end{equation*}
Since $\sigma^2c^2(1-\sigma^2)+\norm{u^j+w^j}^2\leq c^2C_\sigma$ and $s_j\geq\tilde{s}$ for all $j\in\N$, we combine these relations with the definition of $(f+g)_\text{best}^k$ and the above equation to conclude the proof.
\end{proof}

It may seem that the complexity bounds obtained in the above theorem are worse than those derived in Lemma \ref{lem:compl-const} for constant stepsizes, but this is not the case. For instance, if we analyze the complexity for Algorithm \ref{alg:2} with constant stepsize $\alpha$, according to Lemma \ref{lem:compl-const}(ii), if $\epsilon=0$, to achieve a cost function value within $\mathcal{O}(\delta)$ of the optimal, $\alpha$ must be of order $\mathcal{O}(\delta)$. Therefore, the number of necessary iterations needed to guarantee this level of optimality is of order $\mathcal{O}(1/\delta^2)$,
which is the same type of estimate as for Algorithm \ref{alg:2} obtained in Theorem \ref{theo:compl-pol}(ii).

We finish this subsection by analyzing the convergence of Algorithms \ref{alg:1} and \ref{alg:2} using rules \eqref{eq:pol-step} and \eqref{eq:pol-step1}, for the case where the sequence of estimates $(s_k)_{k\in\N}$ satisfies $\lim_{k\to\infty}s_k=\tilde{s}<s^\star$.

\begin{theorem}
Let $(x^k)_{k\in\N}$ be the sequence generated by  Algorithm $\ref{alg:1}$ or $\ref{alg:2}$ with $\alpha_k$ given by \eqref{eq:pol-step} or \eqref{eq:pol-step1}, respectively. Assume that $S^\star\neq\emptyset$, $\epsilon_k\equiv0$ and $\tilde{s}<s^\star$. Then, we have
\begin{equation*}
\lim_{k\to \infty}(f+g)_\text{best}^k \leq s^\star + \dfrac{\overline{\gamma}}{2-\overline{\gamma}}(s^\star-\tilde{s}).
\end{equation*}
\end{theorem}

\begin{proof}
We suppose first that $(x^k)_{k\in\N}$ is calculated by Algorithm \ref{alg:1} with $\alpha_k$ defined by \eqref{eq:pol-step}. We note that $\alpha_k$ can be rewritten as
\begin{equation}
\label{eq:0}
\alpha_k = \overline{\gamma}_k\dfrac{(f+g)(x^k)-s^\star}{\norm{u^k+w^k}^2},\quad\qquad\text{where }\, \, \overline{\gamma}_k := \gamma_k\dfrac{(f+g)(x^k)-s_k}{(f+g)(x^k)-s^\star}.
\end{equation}
Now, we observe that since $\lim_{k\to\infty}s_k=\tilde{s}<s^\star$, we have $(f+g)(x^k)-s_k\geq(f+g)(x^k)-s^\star$ for all $k$ sufficiently large. Therefore, this last inequality, combined with the assumption that $\gamma_k\geq\underline{\gamma}$, implies that $\overline{\gamma}_k\geq\underline{\gamma}$ for all $k$ large enough.

Next, we claim that for all $\gamma<2$ there exists $k(\gamma)\in\N$ such that $\overline{\gamma}_{k(\gamma)}>\gamma$. Indeed, if we suppose by contradiction that for some $\gamma<2$ it holds that $\overline{\gamma}_k\leq\gamma$ for all $k\in\N$, then by the first identity in \eqref{eq:0} and Corollary \ref{cor:1}(ii), we have  $\lim_{k\to\infty}(f+g)(x^k)=s^\star$. This relation, together with the assumption that $\lim_{k\to\infty}s_k=\tilde{s}<s^\star$, now implies that $\overline{\gamma}_k\to\infty$ when $k\to\infty$, arising to a contradiction.

Hence, for all $\epsilon>0$ we have that there exists $\overline{k}\in \N$ such that
\begin{equation*}
\overline{\gamma}_{\overline{k}}
=\gamma_{\overline{k}}\dfrac{(f+g)(x^{\overline{k}})-s_{\overline{k}}}{(f+g)(x^{\overline{k}})-s^\star} > 2- \epsilon,
\end{equation*}
and consequently
\begin{equation*}
\begin{split}
(f+g)(x^{\overline{k}}) < & \, s^\star + \dfrac{\gamma_{\overline{k}}}{2-\epsilon-\gamma_{\overline{k}}}(s^\star-s_{\overline{k}})\\
\leq &\, s^\star + \dfrac{\overline{\gamma}}{2-\epsilon-\overline{\gamma}}(s^\star-\tilde{s}),
\end{split}
\end{equation*}
where the second inequality above follows using that $s_{\overline{k}}\geq\tilde{s}$ and $\gamma_{\overline{k}}\leq\overline{\gamma}$. Since equation above holds for arbitrary $\epsilon>0$, by the definition of $(f+g)_\text{best}^k$ we obtain the result for Algorithm \ref{alg:1}.

Now, if $(x^k)_{k\in\N}$ is the sequence generated by Algorithm \ref{alg:2} with $\alpha_k$ given by \eqref{eq:pol-step1}, then we can rewrite the stepsize $\alpha_k$ as
\begin{equation*}
\alpha_k = \overline{\gamma}_k\dfrac{(f+g)(x^k)-s^\star}{\sigma^2c^2/(1-\sigma)^2+\norm{u^k+w^k}^2},\quad\qquad\text{where }\, \, \overline{\gamma}_k := \gamma_k\dfrac{(f+g)(x^k)-s_k}{(f+g)(x^k)-s^\star}.
\end{equation*}
The rest of the proof runs as before.
\end{proof}

\section{Inexact accelerated PGMs}
\label{sec:acel}

Our goal in this section is to analyze proximal based algorithms for solving problem \eqref{eq:problem} that improve the convergence rate of the methods presented in sections \ref{asbt} and \ref{relat}, while maintaining the inexact calculations of the proximal operators.  
As discussed in section \ref{sec:int}, for this purpose we have to require that one of the functions in \eqref{eq:problem} be differentiable.
Specifically, in the sequel we assume that $\HH=\R^n$, $C=\R^n$ and we also suppose that $f$ has full domain (i.e., $\dom f=\R^n$), it is a differentiable function and $\nabla f$ is Lipschitz continuous, i.e. there is $L > 0$ such that
\begin{equation*}
\norm{\nabla f(x)-\nabla f(y)} \leq L\norm{x-y}
\end{equation*}
for all $x,y\in \dom f$. 

In this section we will discuss inexact accelerated versions of the proximal gradient method (PGM). 
Inexact accelerated PGMs were proposed in \cite{schmidt2011} and \cite{Villa2013}. These works use 
different concepts of approximation of the proximal operator. In \cite{Villa2013}, the authors employ a 
notion of inexactness based on the $\epsilon$-subdifferential, which is indeed a relaxation of the optimality 
condition of the minimization problem in \eqref{eq:prox}. Work \cite{schmidt2011} uses an approximation criterion based on 
an error in the calculation of the proximal objective function. 
To guarantee that these accelerated PGMs achieve the optimal $1/k^2$ convergence rate for the functional values, the errors have to satisfy a sufficiently fast decay condition (see \cite{schmidt2011} and \cite{Villa2013} for more details).

Here we will consider the notions of inexactness presented in Definitions \ref{def:r-approx} and \ref{def:approx-sol2}, in order to approximate the resolvent operators.
For accelerating the PGM using the approximation criterion of Definition \ref{def:r-approx}, it can be used the 
accelerated framework presented in \cite{schmidt2011}. This is due to a relation, which will be proven below, 
that can be established between Definition \ref{def:r-approx} and the concept of inexactness considered in 
\cite{schmidt2011}. 

We now briefly discuss the inexact accelerated framework presented in \cite{schmidt2011} and how to use it to accelerate the PGM when the resolvent operators are inexactly calculated via Definition \ref{def:r-approx}.

Fixing $y\in\R^n$ and $\alpha>0$, a point $x'$ is an $e$-optimal solution of $\textbf{prox}_{\alpha g}(y)$ \cite{schmidt2011}, and we write $x'\approx_{e}\textbf{prox}_{\alpha g}(y)$, if 
\begin{equation}
\label{eq:aprox}
\dfrac{1}{2\alpha} \norm{x'-y}^2 + g(x') \leq e + \min_{x\in\R^n}\{\dfrac{1}{2\alpha}\norm{x-y}^2+g(x)\}.
\end{equation}
The algorithm presented in \cite{schmidt2011} generates sequences $(x^k)_{k\in\N}$ and $(y^k)_{k\in\N}$ by the recursion 
\begin{equation}
\label{eq:accel-1}
\begin{split}
& x^k \approx_{e_k} \textbf{prox}_{\alpha g}(y^{k-1}-\alpha\nabla f(y^{k-1}))\\
& y^k = x^k + \beta_k(x^k-x^{k-1}),
\end{split}
\end{equation}
where the stepsize $\alpha=1/L$ is fixed and $(\beta_k)_{k\in\N}$ is the sequence defined by $\beta_k=k-1/k-2$, for all $k$. In order to this method achieve the optimal $\mathcal{O}(1/k^2)$ rate, it is necessary  that the sequence $(k\sqrt{e_k})_{k\in\N}$ be summable (see \cite[Proposition 2]{schmidt2011}).

The following lemma gives a link between approximations in the sense of Definition \ref{def:r-approx} with those in the sense of \eqref{eq:aprox}.

\begin{lemma}
Given $y\in\R^n$, $\alpha>0$ and $r>0$. If $x'$ is an $r$-approximate solution of $\textbf{prox}_{\alpha g}$ at $y$, then $x'\approx_e\textbf{prox}_{\alpha g}(y)$, where $e=(1/2\alpha)r$.
\end{lemma}

\begin{proof}
Since $x'$ is an $r$-approximate solution of $\textbf{prox}_{\alpha g}$ at $y$, there exists $w\in\partial g(x')$ such that $\norm{\alpha w + x' - y}\leq r$. From the definition of the subdifferential we have that
\begin{equation*}
g(x') + \dfrac{1}{2\alpha}\norm{x'-y}^2 \leq g(z) + \inpr{w}{x'-z} + \dfrac{1}{2\alpha}\norm{x'-y}^2,
\end{equation*}
where $z := \textbf{prox}_{\alpha g}(y)$. Now, adding and subtracting $(1/2\alpha)\norm{z-y}^2$ on the right-hand side of the inequality above, we obtain
\begin{equation*}
\begin{split}
g(x') + \dfrac{1}{2\alpha}\norm{x'-y}^2 \leq g(z) + \dfrac{1}{2\alpha}\norm{z-y}^2 + \dfrac{1}{2\alpha}\left[2\alpha\inpr{w}{x'-z} + \norm{x'-y}^2 - \norm{z-y}^2\right].
\end{split}
\end{equation*}
Form the definition of $\textbf{prox}_{\alpha g}$ in \eqref{eq:prox} and noting that it does not change if we multiply by $1/\alpha$, we have that 
\begin{equation*}
g(z) + \dfrac{1}{2\alpha}\norm{z-y}^2 = \min_{x\in\R^n} \{\dfrac{1}{2\alpha}\norm{x-y}^2+g(x)\}.
\end{equation*}
Next, we observe that 
\begin{equation*}
\begin{split}
\norm{x'-y}^2 - \norm{z-y}^2 = & \norm{x-z}^2 + 2\inpr{z-y}{x'-z}\\
= & \norm{x-z}^2 - 2\alpha\inpr{v}{x'-z},
\end{split}
\end{equation*}
where $v\in\partial g(z)$ is such that $z+\alpha v = y$. Combining the three equations above we obtain the desired estimate
\begin{equation*}
\begin{split}
g(x') + \dfrac{1}{2\alpha}\norm{x'-y}^2 \leq & \min_{x\in\R^n} \{\dfrac{1}{2\alpha}\norm{x-y}^2+g(x)\} + \dfrac{1}{2\alpha}\left[\norm{x-z}^2 + 2\alpha\inpr{w-v}{x'-z}\right]\\
\leq & \min_{x\in\R^n} \{\dfrac{1}{2\alpha}\norm{x-y}^2+g(x)\} + \dfrac{1}{2\alpha}r,
\end{split}
\end{equation*}
where the second inequality above follows from \eqref{eq:nova}.
\end{proof}

According to the above lemma, if we replace in \eqref{eq:accel-1} the condition that $x^k$ is an $e_k$-optimal solution by $x^k$ is an $r_k$-approximate solution, we have an inexact accelerated PGM that uses the approximation criterion in Definition \ref{def:r-approx} to evaluate the proximal operators inexactly. By the 
convergence analysis of \cite{schmidt2011}, this inexact accelerated PGM achieves the convergence rate of $\mathcal{O}(1/k^2)$, provided that the sequence $(k\sqrt{r_k})_{k\in\N}$ is summable.

We now focus on accelerating the PGM using a relative error criterion for approximating the resolvent operator. We observe that we were not able to establish the convergence of an inexact accelerated PGM with the criterion in Definition \ref{def:approx-sol}. Instead, we will consider the notion of inexactness presented in Definition \ref{def:approx-sol2}.
Algorithm \ref{alg:4} below is based on the ideas presented in \cite{MonSvaAHPE}, where inexact accelerated hybrid extragradient proximal methods were proposed.

\vspace{5mm}

\fbox{\parbox{14.3cm}{
\begin{algorithm}\label{alg:4}
Choose $x^0,\,\overline{x}^0\in\R^n$ and $\sigma^2\in\left[0,1/2\right[$. Set $t_0=0$ and $\alpha=\sigma^2/L$. Then, for all $k=1,2,\dots$
\begin{enumerate}
\item[1.] Define
\begin{equation*}
\begin{split}
& \beta_k = \dfrac{\alpha(1-\sigma^2)+\sqrt{\alpha^2(1-\sigma^2)^2+4\alpha(1-\sigma^2)t_{k-1}}}{2}\\
& t_k = t_{k-1} + \beta_k\\
& \tilde{x}^k = \dfrac{t_{k-1}}{t_{k}}\overline{x}^{k-1} + \dfrac{\beta_k}{t_{k}}x^{k-1}
\end{split}
\end{equation*}
and set $y^k = \tilde{x}^k-\alpha\nabla f(\tilde{x}^k)$.

\item[2.] Compute a triplet $(\overline{x}^k,\overline{w}^k,\overline{\epsilon}_k)\in\R^n\times\R^n\times\R_+$ such that 
\begin{equation}
\label{eq:err-accel}
\left\lbrace
\begin{array}{l}
\overline{w}^k\in\partial_{\overline{\epsilon}_k}g(\overline{x}^k),\\
\norm{\alpha\overline{w}^k+\overline{x}^k-y^k}^2 + 2\alpha\overline{\epsilon}_k \leq \sigma^2(\norm{\overline{x}^k-\tilde{x}^k}^2 + \norm{\alpha(\overline{w}^k+\nabla f(\tilde{x}^k))}^2).
\end{array}
\right.
\end{equation}

\item[3.] Set
\begin{equation*}
\begin{split}
& x^k = x^{k-1} - \beta_k(\nabla f(\tilde{x}^k)+\overline{w}^k).
\end{split}
\end{equation*}
\end{enumerate}
\end{algorithm}}}

\vspace{4mm}

We make several remarks on Algorithm \ref{alg:4}. 
First, we note that the maximum tolerance for the relative error in the resolution of \eqref{eq:err-accel} is 1/2, instead of 1 as in Definition \ref{def:approx-sol2}. 
Second, we observe that if for all $k\geq1$ we define $\epsilon_k=f(\overline{x}^k)-f(\tilde{x}^k)-\inpr{\nabla f(\tilde{x}^k)}{\overline{x}^k-\tilde{x}^k}$, then by Proposition \ref{prop:sub-properties}(ii) we have that $\nabla f(\tilde{x}^k)\in\partial_{\epsilon_k} f(\overline{x}^k)$, and since $\nabla f$ is $L$-Lipschitz, it also holds that
\begin{equation}
\label{eq:12}
\epsilon_k \leq \dfrac{L}{2}\norm{\tilde{x}^k-\overline{x}^k}^2
\end{equation}
(see for instance \cite{Nest2004}). 

Third, from Proposition \ref{prop:sub-properties}(i) and the previous observations it follows that $\overline{w}^k + \nabla f(\tilde{x}^k)\in\partial_{\overline{\epsilon}_k+\epsilon_k}(f+g)(\overline{x}^k)$. Therefore, by step 2 of Algorithm \ref{alg:4} and rearranging the terms in the left-hand side of the inequality in \eqref{eq:err-accel}, we have that the triplet $(\overline{x}^k,\overline{w}^k+\nabla f(\tilde{x}^k),\overline{\epsilon}_k+\epsilon_k)$ satisfy the Definition \ref{def:approx-sol2} of approximate solution of \eqref{eq:proximal-system} (with $f+g$) at $(\alpha,\tilde{x}^k)$.

Fourth, if we define the \emph{linear approximation} of $f+g$ at $\overline{x}^k$, for all $k\geq1$, i.e. the function $\ell_k:\R^n\to\left]-\infty,\infty\right]$
\begin{equation*}
\ell_k(x) := (f+g)(\overline{x}^k) + \inpr{\overline{w}^k+\nabla f(\tilde{x}^k)}{x-\overline{x}^k} - \overline{\epsilon}_k - \epsilon_k\qquad \forall x\in\R^n,
\end{equation*}
and the affine functions $\Ell_0\equiv0$ and $\Ell_k := \dfrac{1}{t_k}\sum\limits_{i=1}^k\beta_i\ell_i$ for all $k\geq1$.
Then, for all $k\geq0$, we have
\begin{equation}
\label{eq:n1}
\begin{split}
\ell_{k+1}\leq(f+g),\qquad t_k\Ell_k\leq t_k(f+g)\quad \text{and} \quad x^k=\arg\min_{x\in\R^{n}}\{t_k\Ell_k(x)+1/2\norm{x-x^0}^2\}.
\end{split}
\end{equation}

In order to establish the convergence rate of Algorithm \ref{alg:4} we need the following technical result, which is in the spirit of \cite[Lemma 3.3]{MonSvaAHPE}. For the sake of completeness, we include its proof here.

\begin{lemma}
For all $x\in\R^n$ and $k\geq1$ the following inequality holds  
\begin{equation}
\label{eq:7}
\ell_k(x) + \dfrac{1}{2\alpha(1-\sigma^2)}\norm{x-\tilde{x}^k}^2\geq (f+g)(\overline{x}^k) + \dfrac{1-2\sigma^2}{2\alpha}\norm{\overline{x}^k-\tilde{x}^k}^2.
\end{equation}
\end{lemma}

\begin{proof}
We consider the following minimization problem 
\begin{equation}
\label{eq:10}
\min_{x\in\R^n} \{\ell_k(x) + \dfrac{1}{2\alpha(1-\sigma^2)}\norm{x-\tilde{x}^k}^2\},
\end{equation}
and observe that the first-order optimality condition for this convex optimization problem yields that
its minimizer $x'$ satisfies 
\begin{equation*}
0 =  \overline{w}^k+\nabla f(\tilde{x}^k) + \dfrac{1}{\alpha(1-\sigma^2)} (x' - \tilde{x}^k).
\end{equation*}
Therefore, the solution of \eqref{eq:10} is $x' = \tilde{x}^k-\alpha(1-\sigma^2)(\overline{w}^k+\nabla f(\tilde{x}^k))$. Substituting $x'$ into the function in the minimization problem \eqref{eq:10} we have, for all $x\in\R^n$, that
\begin{equation*}
\begin{split}
\ell_k(x) + \dfrac{1}{2\alpha(1-\sigma^2)}\norm{x-\tilde{x}^k}^2 \geq & (f+g)(\overline{x}^k) + \inpr{\overline{w}^k+\nabla f(\tilde{x}^k)}{\tilde{x}^k-\alpha(1-\sigma^2)(\overline{w}^k+\nabla f(\tilde{x}^k))-\overline{x}^k}\\
& -\overline{\epsilon}_k - \epsilon_k + \dfrac{1-\sigma^2}{2\alpha}\norm{\alpha(\overline{w}^k+\nabla f(\tilde{x}^k))}^2\\
= & (f+g)(\overline{x}^k) + \inpr{\overline{w}^k+\nabla f(\tilde{x}^k)}{\tilde{x}^k-\overline{x}^k} - \overline{\epsilon}_k - \epsilon_k\\
& - \dfrac{1-\sigma^2}{2\alpha}\norm{\alpha(\overline{w}^k+\nabla f(\tilde{x}^k))}^2.
\end{split}
\end{equation*}

Now, we note that 
\begin{equation*}
\begin{split}
\inpr{\overline{w}^k+\nabla f(\tilde{x}^k)}{\tilde{x}^k-\overline{x}^k} - \overline{\epsilon}_k - \epsilon_k = & \dfrac{1}{2\alpha}\norm{\overline{x}^k-\tilde{x}^k}^2 + \dfrac{1}{2\alpha}\norm{\overline{w}^k+\nabla f(\tilde{x}^k)}^2\\
&  - \dfrac{1}{2\alpha}\left[\norm{\overline{w}^k+\overline{x}^k-y^k}^2 + 2\alpha\overline{\epsilon}_k + 2\alpha\epsilon_k\right]\\
\geq & \dfrac{1}{2\alpha}\norm{\overline{x}^k-\tilde{x}^k}^2 + \dfrac{1}{2\alpha}\norm{\overline{w}^k+\nabla f(\tilde{x}^k)}^2\\
& - \dfrac{\sigma^2}{2\alpha}\left[\norm{\overline{x}^k-\tilde{x}^k}^2 + \norm{\alpha(\overline{w}^k+\nabla f(\tilde{x}^k))}^2 + \norm{\overline{x}^k-\tilde{x}^k}^2\right],
\end{split}
\end{equation*}
where the inequality above follows from the error criterion in step 2 of Algorithm \ref{alg:4}, \eqref{eq:12} and the definition of $\alpha$.
Combining the two relations above we conclude.
\end{proof}

The following result yields the convergence rate of the accelerated Algorithm \ref{alg:4}. The proof of this 
convergence result is similar to the proof of Lemma 3.4 of \cite{MonSvaAHPE}. For the sake of completeness, we 
provide its proof in the Appendix.

\begin{proposition} 
\label{prop:accel}
The sequences $(t_k)_{k\in\N}$, $(\overline{x}^k)_{k\in\N}$ and $(\tilde{x}^k)_{k\in\N}$ generated by Algorithm $\ref{alg:4}$ satisfy the following inequality for any $k\geq1$
\begin{equation}
\label{eq:9}
t_k\geq\dfrac{k^2\sigma^4(1-\sigma^2)}{4L},\qquad\qquad t_k(f+g)(\overline{x}_k)  \leq t_k\Ell_k(x) + \dfrac{1}{2}\norm{x-x^0}^2.
\end{equation}
\end{proposition}

\begin{corollary}
For every integer $k\geq1$, we have
\begin{equation*}
(f+g)(\overline{x}_k) - s^\star \leq \dfrac{2Ld_0^2}{\sigma^4(1-\sigma^2)k^2}.
\end{equation*}
\end{corollary}

\begin{proof}
The second inequality in \eqref{eq:9} with $x=x^\star$, the projection of $x_0$ into $S^\star$, yields
\begin{equation*}
t_k(f+g)(\overline{x}_k)  \leq t_k\Ell_k(x^\star) + \dfrac{1}{2}d_0^2.
\end{equation*}
Combining equation above with the second relation in equation \eqref{eq:n1} and the first inequality in \eqref{eq:9} we conclude.
\end{proof}

\section{Numerical experiments}\label{numer}

This section describes some simple, preliminary computational testing of the P$\epsilon$SMs. These experiments are not meant to be exhaustive, but only to provide a preliminary indication of the performance of these methods. 
First, we will test Algorithm \ref{alg:1} (P$\epsilon$SM1) in a classic $\ell_1-$regularization problem. Then, we will illustrate the performance of Algorithm \ref{alg:2} (P$\epsilon$SM2) when solving the Total Variation deblurring problem \cite{Villa2013,BecTeb09}.

\subsection{Example for Algorithm \ref{alg:1}}

In this section we show some numerical experiments to test the {\bf P$\epsilon$SM1} and compare it with \cite[Algorithm PSS]{BelloCruz2016}.  
 We are interested in implementing the  classic $\ell_1-$regularization for the quadratic minimization problem, which consists to solve the problem \eqref{eq:problem} with $f, g:\RR^n\rightarrow \RR$ defined by $f(x)=\frac{1}{2}\|Ax-b\|^2$ and $g(x)=\|x\|_1$, being $A\in \mathbb{M}^{n\times n}(\RR)$ a semi-definite positive random matrix, $b\in \RR^n$ is a vector with all entries equal to $1$ and for $x=(x_1,x_2,\cdots,x_n)$, define $\|\cdot\|_1$ by $\|x\|_1:=\sum_{i=1}^n|x_i|$. 
 
 For the implementation, we used the stopping criterion $\|x^{k+1}-x^k\|^2\leq 10^{-4}$, with $(x^k)_{k\in\NN}$ generated by the algorithm. For all $k\in\NN$ we utilized $\alpha_k=\frac{\alpha_0}{k}$ with $\alpha_0=\frac{1}{\|A\|^2}$, $\epsilon_k=1/k$ and $r_k=1/k$. For ``funcVal" we denote the value of the function $h(x)=f(x)+g(x)$ once the process stop, {\it `` iter"} denotes the number of iteration of the algorithm. By the initial point we take $x^0=(1,1,\cdots, 1)$. The computation of the inexact proximal with absolute error (step 2 in the P$\epsilon$SM1), was computed by using a line-search between the exact proximal and the point $y^k$ defined in step 1. 

We use MATLAB version R2016b on a PC  with Intel(R) Core(TM) i7-7500 CPU @ 2.90GHz and Windows 10 Home. The results for the example are listed in {\bf Table \ref{tab:1}}.

\begin{table}[H]
\begin{center}
\begin{tabular}{ |p{1cm}| p{1.7cm} p{1.8cm} p{1.5cm} |p{1.7cm} p{1.8cm} p{1.5cm}|  }
 \hline
 & {\bf P$\epsilon$SM1}  & &   &{\bf Alg PSS} &in \cite{BelloCruz2016}&  \\
 \hline
 $n$   & iter   & CPU time  & FuncVal & iter & CPU time & FuncVal \\
 \hline
 1& 14& 0.390625 &0.52036& 2&0.046875 &0.5\\
 5& 17& 0.03125 &1.88179& 10 & 0.015625 &1.97205\\
 10& 3 &0.015625 &1.89058& 3  & 0.03125& 1.95775 \\
 20&3 &0.015625 & 1.94297& 3  &0.015625 & 2.2872  \\
 40& 3&0.015625& 5.45126 &2 &0.015625&11.0227\\
 80& 3 &0.0625 &25.767 & 2 &0.03125 &85.6772 \\
 100& 3 &0.03125& 50.2412 & 2& 0.046875 & 175.411 \\
 200& 2 &0.015625& 768.378 &2& 0.046875 & 1348.37 \\
 500& 2 &0.078125& 12109 &2& 0.03125 & 21407.1 \\
 1000& 2 &0.140625& 96221.2 &2& 0.046875 & 170555 \\
 5000& 2 &9.67188& 1.204e+07 &2& 0.234375 & 2.140e+07 \\
 10000& 2 &94.9375& 9.64e+07 &2& 0.75 & 1.71e+08 \\
 \hline
\end{tabular}
\end{center}
\caption{Results for the Example}
\label{tab:1}
\end{table}

As the results show, we can see that the P$\epsilon$SM1 has competitive results compared with Algorithm PSS in \cite{BelloCruz2016}. The main advantage relies on the function-valued once the algorithm stops. For higher dimensional spaces, greater than 1000, the CPU time of our algorithm is worst than the compared one. In compensation, the value of the function is much better, which is an improvement. 

\subsection{Example for Algorithm \ref{alg:2}}
\label{sub:numer2}

In the sequel, an application of the P$\epsilon$SM2 to an image deblurring problem is considered. 
The main goal of this experiment is to illustrate the performance of the relative error criterion considered in the P$\epsilon$SM2 for approximately calculating the proximal operator.

The Total Variation deblurring problem seeks to estimate an unknown original image $x\in\R^{N\times N}$ from an observed blurred image $b\in\R^{N\times N}$, by solving the minimization problem
\begin{equation}
\label{eq:tv deblur}
\min_{x\in\R^{N\times N}}\, \frac{1}{2}\norm{Ax-b}^2_F + \tau\sum_{i,j=1}^N \norm{(\nabla x)_{i,j}}_2,
\end{equation}
where $\nabla: \R^{N\times N} \to \R^{N\times N} \times \R^{N\times N}$ is the \textit{discrete gradient} operator (see \cite{Cha04} for the precise definition), $\tau>0$ is a regularization parameter, $\norm{\cdot}_2$ is the Euclidean norm in $\R^2$ and $A:\R^{N\times N}\to\R^{N\times N}$ is a linear operator that represents some blurring operator.

To solve the problem \eqref{eq:tv deblur} via the P$\epsilon$SM2, we take $C=\R^{N\times N}$, $f(x)=\frac{1}{2}\norm{Ax-b}^2_F$ and $g(x)=\tau\sum_{i,j=1}^N \norm{(\nabla x)_{i,j}}_2$. In step 1, since $f$ is differentiable, we take $\epsilon_k=0$ and $u^k=\nabla f(x^k)$ for all iteration $k$.
In order to obtain a triplet $(\overline{x}^k,\overline{w}^k,\overline{\epsilon}_k)$ such that the relative error condition in step 2 of the P$\epsilon$SM2 is met, we use the dual approach presented in \cite{Villa2013}, which we briefly discuss here. 

First, we observe that $g(x)=\omega(\nabla x)$, where $\omega:\R^{N\times N} \times \R^{N\times N}\to\R$, $\omega(p)=\tau\sum_{i,j=1}^N \norm{p_{i,j}}_2$. Also, we note that by definition, for finding $\textbf{prox}_{\alpha g}(y)$ it is necessary to solve the optimization problem 
\begin{equation}
\label{eq:num2}
\min_{x\in\HH} \Phi_\alpha(x),
\end{equation}
where $\Phi_\alpha(x)=g(x) + \frac{1}{2\alpha}\norm{x-y}^2$. The \emph{dual function} of \eqref{eq:num2} is defined as $\Psi_\alpha(v)=\frac{1}{2\alpha}\norm{\alpha\nabla^*v-y}^2+\omega^*(v)-\frac{1}{2\alpha}\norm{y}^2$, where $\omega^*$ is the \emph{Fenchel-Legendre conjugate} \cite{lemarechal_2} of $\omega$, and the \emph{duality gap} is $G(x,v)=\Phi_\alpha(x)+\Psi_\alpha(v)$.
In \cite[Proposition 2.2]{Villa2013}, it was proved that if
\begin{equation}
\label{eq:num1}
G(y-\alpha\nabla^*v,v)\leq\frac{\epsilon^2}{2\alpha},\quad \text{then}\quad \nabla^*v\in\partial_{\frac{\epsilon^2}{2\alpha}}g(y-\alpha\nabla^*v).
\end{equation} 
Moreover, if $(v^l)_{l\in\N}$ is a minimizing sequence for $\Psi_\alpha$, then $z^l=y-\alpha\nabla^*v^l$ is such that $z^l\to\textbf{prox}_{\alpha g}(y)$ and $G(z^l,v^l)\to0$ (see \cite[Theorem 6.1]{Villa2013}).

Now, we observe that if at each iteration $k$ we have a sequence $(v^l)_{l\in\N}\subset\R^{N\times N}$ and we take $\overline{x}^l=y_k-\alpha_k\nabla^*v^l$, $\overline{w}^l=\nabla^*v^l$ and $\overline{\epsilon}_l=G(y_k-\alpha_k\nabla^*v_l,v_l)$, then from \eqref{eq:num1} it follows that $\overline{w}^l\in\partial_{\overline{\epsilon}_l}g(\overline{x}^l)$. Therefore, for each $l$ we have an available triplet $(\overline{x}^l,\overline{w}^l,\overline{\epsilon}_l)$ to test the error condition in step 2 of the P$\epsilon$SM2. Moreover, if $(v^l)_{l\in\N}$ is a minimizing sequence for the dual function $\Psi_{\alpha_k}$, then, whenever $y^k$ is not a fixed point of $\textbf{prox}_{\alpha_kg}$, we have that there exists a finite $l(k)$ such that $(\overline{x}^{l(k)},\overline{w}^{l(k)},\overline{\epsilon}_{l(k)})$ satisfies the inequality in step 2 of the P$\epsilon$SM2.

Since $C=\R^{N\times N}$, the update rule in step 3 of the P$\epsilon$SM2 is reduced to $x^{k+1}=y^k-\alpha_k\overline{w}^k=\overline{x}^k$, where the second equality is due to the choice of $\overline{x}^k$ and $\overline{w}^k$ discussed above. 

To provide a comparison, we also implemented the inexact proximal gradient method (IPGM) proposed in \cite{schmidt2011}, which is the method described in \eqref{eq:accel-1} with $\beta_k=0$ for all $k$.  In order to obtain at each iteration $k$ an approximate solution of $\textbf{prox}_{\alpha_kg}$ in the sense of \eqref{eq:aprox}, we can also use the dual approach of \cite{Villa2013}. Indeed, if $v$ is such that the first inequality in \eqref{eq:num1} holds, then taking $x'=y-\alpha\nabla^*v$, from the inclusion in \eqref{eq:num1} we have that equation \eqref{eq:aprox} holds with $e=\frac{\epsilon^2}{2\alpha}$. Thus, if at each iteration $k$ we have a minimizing sequence $(v^l)_{l\in\N}$ for $\Psi_{\alpha_k}$, for a given $e_k>0$ there exists a finite $l(k)$ such that $x^{l(k)}=y^k-\alpha_k\nabla^*v^{l(k)}$ is an $e_k$-optimal solution of $\textbf{prox}_{\alpha_kg}(y^k)$\footnote{We stress that, for the IPGM, $y^k=x^{k-1}-\alpha_k\nabla f(x^{k-1})$ at each iteration $k$.}.

We observe that, due to the choice of $\overline{x}^k$ and the fact that $f$ is differentiable and $C=\R^{N\times N}$, the only difference between the P$\epsilon$SM2 and the IPGM to solve the problem \eqref{eq:tv deblur} is how the approximation of the proximal operator is calculated, and this is the dominant computation at each iteration of both methods. Therefore, these experiments are actually comparing the performances of the relative error criterion and the absolute error criterion considered, respectively, in each method.

We implemented both methods in Matlab code. In our implementation, at each iteration $k$ we use the \emph{fast iterative shrinkage-thresholding} algorithm (FISTA) of \cite{fista} to minimize the dual function $\Psi_{\alpha_k}$, as suggested in \cite{Villa2013} (see (6.8) of \cite{Villa2013}). 
For the P$\epsilon$SM2, we stop the inner method when either the relative error condition of step 2 is satisfied or it reaches a maximum number of iterations. For the IPGM, we choose a sequence $(e_k)_{k\in\N}$ and, at each iteration $k$, we stop the inner algorithm when either the duality gap is less than $e_k$ or it reaches a maximum number of iterations. According to the convergence analysis of \cite{schmidt2011}, we consider sequences of errors of type $\sqrt{e_k}=\mathcal{O}(1/k^q)$ with $q>1$, which guarantees the convergence of the IPGM, see \cite[Proposition 1]{schmidt2011} and the comments below. We set as $3000$ the maximum number of iterations for FISTA, in both algorithms.

In the numerical simulations, we consider the $256\times256$ \emph{Cameraman} test image blurred by a $4\times4$ Gaussian blur with standard deviation $2$ followed by an additive normal noise with zero mean and standard deviation $10^{-4}$. The regularization parameter $\tau$ was set to $10^{-4}$. We fix $\alpha_k=\alpha=1/L$, where $L=\norm{A^*A}$, and we choose $x^0=b$. The stopping criterion for the P$\epsilon$SM2 and IPGM is that they either reach $5000$ iterations or the relative difference $\norm{x^k-x^{k-1}}/\norm{x^k}$ is less than $10^{-4}$. We test the P$\epsilon$SM2 for various values of $\sigma^2$ chosen between $0.1$ and $0.9$. For the IPGM we consider two sequences of error, $\sqrt{e_k}=1/k^q$ and $
\sqrt{e_k}=C/k^q$, with $q$ chosen between $1.1$ and $1.9$. In the computational experiments, we refer to the variant of IPGM with $\sqrt{e_k}=1/k^q$ as IPGM1, and with $\sqrt{e_k}=C/k^q$ as IPGM2. The coefficient $C$ is chosen by solving the equation $G(y^0-\alpha\nabla f(y^0),
0)=C^2/(2\alpha)$, following \cite{Villa2013}.

Table \ref{tab:2} displays the results of the experiments. The columns mean: Method, the algorithm used (P$\epsilon$SM2 with the corresponding value of $\sigma$ and IPGM with the corresponding sequence of error and value of $q$); RelDiff, the relative difference at the last iterate; CPU time, the time (in seconds) needed  to  reach the desired accuracy for the relative difference; FuncVal, the function value at the last iterate; ExtIt, the number of external iterations; and IntIt, the total number of inner iterations.

\begin{table}[H]
\begin{center}
\begin{tabular}{c|ccr|cr}
\hline
Method & RelDiff & FuncVal & CPU time & ExtIt & IntIt\\
\hline
P$\epsilon$SM2 &   &  &  &  & \\
$\sigma^2=0.9$ & $0.00009950$	& $0.36002654$ & $151.6111$ & $172$ &	$55759$\\
\hline
$\sigma^2=0.7$ & $0.00009951$ & $0.35998493$ & $165.3417$ & $172$ & $59823$\\
\hline
$\sigma^2=0.5$ & $0.00009952$ & $0.35995523$ & $181.8645$ & $172$ & $65857$\\
\hline
$\sigma^2=0.3$ & $0.00009954$ & $0.35994516$ & $215.6910$ & $172$ & $76445$\\
\hline
$\sigma^2=0.1$ & $0.00009955$ & $0.35987295$ & $306.3176$ & $172$ & $107035$\\
\hline
IPGM1: $\sqrt{e_k}=1/k^q$ &    &   &  &  & \\
$q=1.1$ & $0.00009956$ & $0.35986216$ & $336.4117$ & $172$ & $126437$\\
\hline
$q=1.3$ & $0.00009956$ & $0.35984462$ & $795.8492$ & $172$ & $288798$\\
\hline
$q=1.5$ & $0.00009957$ & $0.35985049$ & $1039.2948$ & $172$ & $379107$\\
\hline
$q=1.7$ & $0.00009957$ & $0.35984905$ & $1200.7429$ & $172$ & $422907$\\
\hline
$q=1.9$ & $0.00009957$ & $0.35987493$ & $1423.3937$ & $172$ & $455606$\\
\hline
IPGM2: $\sqrt{e_k}=C/k^q$ &    &   &  &  & \\
$q=1.1$ & $0.00009956$ & $0.35991933$ & $137.6724$ & $172$ & $49961$\\
\hline
$q=1.3$ & $0.00009955$ & $0.35986663$ & $252.2872$ & $172$ & $91017$ \\
\hline
$q=1.5$ & $0.00009956$ & $0.35986485$ & $669.2156$ & $172$ & $232471$ \\
\hline
$q=1.7$ & $0.00009957$ & $0.35987348$ & $986.3419$ & $172$ & $340939$\\
\hline
$q=1.9$ & $0.00009956$ & $0.35984404$ & $1244.5066$ & $172$ & $395763$\\
\hline
\end{tabular}
\end{center}
\caption{Results for the Total Variation deblurring problem.}
\label{tab:2}
\end{table}

We observe that all methods executed the same number of external iterations and they reached similar functional values. We also observe that the fastest method was the IPGM2 with $q=1.1$. However, it is not much faster than the P$\epsilon$SM2 with $\sigma^2=0.9$. In addition, we note that as we increase the precision for the error criterion of the P$\epsilon$SM2, the time to reach the desired accuracy gradually increases.  Furthermore, the P$\epsilon$SM2 with the highest accuracy ($\sigma^2=0.1$) is just two times slower than with the smallest accuracy ($\sigma^2=0.9$). In contrast, the performance of the IPGM depends greatly on the precision for the error criterion. Indeed, we observe that the processing time of the IPGM2 with accuracy $q=1.3$ is almost twice the processing time with accuracy $q=1.1$, and the IPGM2 with the highest accuracy ($q=1.9$) is at least nine times slower than with the smallest accuracy ($q=1.1$). 

This difference in performance is less evident in the IPGM1 where, while the processing time 
with $q=1.3$ is also twice the processing time with $q=1.1$, the processing time with the highest accuracy is less than five times the processing time with the smallest accuracy. However, as 
we observed in the numerical experiments, this was due to the fact that for the IPGM1 
with greater precisions, in the majority of iterations, FISTA was stopping because it reached the maximum 
number of iterations instead of by the stopping rule given by the duality gap. Indeed, for 
instance, we observe that the ratio between the internal and external iterations for $q=1.9$ is 
$2648$ and for $q=1.7$ is $2458$. Thus, these experiments suggest that the relative 
error criterion of the P$\epsilon$SM2 is less sensitive to the accuracy of the error, while 
maintaining the overall efficiency of the method.

Finally, we would like to point out the difference in performance between the IPGM1 and the IPGM2. Whereas the IPGM2 with $q=1.1$ was the fastest method, the IPGM1 with the smallest accuracy $q=1.1$ is slow, even slower than the P$\epsilon$SM2 with the highest precision $\sigma^2=0.9$. This evidences the great sensitivity of the IPGM to the choice of the sequence of errors for the absolute error criterion, which must be chosen in advance.

\section{Conclusions}\label{sec:con}
In this paper, we present three inexact algorithms for solving the constrained optimization problem when the objective function is the sum of two convex functions. The main difference between the algorithms is a way of computing the inexact proximal operator. Also, for relaxing the known schemes in the literature, we consider the enlargement of the subdifferential in each case. The first two methods, considered in Hilbert spaces, present three different strategies for choosing the step size, with particular behavior in each case. The last algorithm, considered in the finite-dimensional Euclidean spaces, is an accelerated version, in the sense of Nesterov's scheme, for improving the convergence properties. In all cases, the convergence analysis is studied. 

Our preliminary numerical experiments show that our algorithms have a competitive performance with respect to similar algorithms in the literature and suggest an advantage in the use of relative error criteria for computing approximations of the proximal operator. To confirm these hypotheses, the next step should be to experiment with a larger sample of problems.
 
The motivation of relaxing the existing methods in the literature comes from the large applicability in practical problems. 

\section*{Acknowledgments}
This work was partially completed while M.P.M was supported by a CAPES post-doctoral fellowship at the University of Campinas. M.P.M is very grateful to IMECC at the University of Campinas and especially to Professor Sandra Augusta Santos for the warm hospitality.

\section*{Appendix}
\noindent {\it Proof of Proposition \ref{prop:accel}.}
We first define, for all $k\geq0$,
\begin{equation}
\label{eq:6}
\eta_k : = \min_{x\in\R^n}\{t_k\Ell_k(x)+1/2\norm{x-x^0}^2\},
\end{equation}
and prove that $\eta_{k+1}-\eta_k\geq t_{k+1}(f+g)(\overline{x}^{k+1}) - t_k(f+g)(\overline{x}^k)$.
To do this, we observe that 
\begin{equation}
\label{eq:13}
\begin{split}
\eta_{k+1} = & t_{k+1}\Ell_{k+1}(x^{k+1}) + \dfrac12\norm{x^{k+1}-x^0}^2\\
 = & \beta_{k+1} \ell_{k+1}(x^{k+1}) + t_{k}\Ell_{k}(x^{k+1}) +  \dfrac12\norm{x^{k+1}-x^0}^2,
\end{split}
\end{equation}
where the first equality above follows from the last relation in \eqref{eq:n1}, and the second equality is a consequence of the definition of $\Ell_k$.

Next, we note that the definition of $\eta_k$, together with the fact that the function in the minimization problem in \eqref{eq:6} is quadratic, implies that
$$t_{k}\Ell_{k}(x^{k+1}) +  \dfrac12\norm{x^{k+1}-x^0}^2 = \eta_{k}  + \dfrac12\norm{x-x^{k}}^2.$$ 
Therefore, combining this latter relation with \eqref{eq:13} we obtain
\begin{equation}
\label{eq:n2}
\begin{split}
\eta_{k+1} = & \beta_{k+1}\ell_{k+}(x^{k+1}) + \eta_{k}  + \dfrac12\norm{x-x^{k}}^2\\
= & \eta_{k} - t_{k}(f+g)(\overline{x}^{k}) + \beta_{k+1}\ell_{k+1}(x^{k+1}) + t_{k}(f+g)(\overline{x}^{k})\\
& + \dfrac{1}{2}\norm{x^{k+1}-x^{k}}^2\\
\geq & \eta_{k} - t_{k}(f+g)(\overline{x}^{k}) + \beta_{k+1}\ell_{k+1}(x^{k+1}) + t_{k}\ell_{k+1}(\overline{x}^{k})\\
& + \dfrac{1}{2}\norm{x^{k+1}-x^{k}}^2,
\end{split}
\end{equation}
where the inequality above is due to the first relation in \eqref{eq:n1}.

Now, from the definition of $t_{k+1}$ and because $\ell_{k+1}$ is affine, we have 
\begin{equation*}
\beta_{k+1}\ell_{k+1}(x^{k+1}) + t_{k}\ell_{k+1}(\overline{x}^{k}) = t_{k+1}\ell_{k+1}\left(\dfrac{\beta_{k+1}}{t_{k+1}}x^{k+1} + \dfrac{t_k}{t_{k+1}}\overline{x}^k\right).
\end{equation*}
Moreover, denoting $x = \dfrac{\beta_{k+1}}{t_{k+1}}x^{k+1} + \dfrac{t_k}{t_{k+1}}\overline{x}^k$ and using the definition of $\tilde{x}^{k+1}$ in step 1 of Algorithm \ref{alg:4}, we have $x^{k+1} - x^k = \dfrac{t_{k+1}}{\beta_{k+1}}(x-\tilde{x}^{k+1})$. Therefore, combining these relations with equation \eqref{eq:n2} we obtain
\begin{equation*}
\begin{split}
\eta_{k+1} \geq & \eta_k -t_k(f+g)(\overline{x}^k) + t_{k+1}\ell_{k+1}(x) + \dfrac{t_{k+1}^2}{2\beta_{k+1}^2}\norm{x-\tilde{x}^{k+1}}^2\\
= & \eta_k - t_k(f+g)(\overline{x}^k) + t_{k+1}\left(\ell_{k+1}(x) + \dfrac{1}{2\alpha(1-\sigma^2)}\norm{x-\tilde{x}^{k+1}}^2\right),
\end{split}
\end{equation*}
where we used in the equality above the definition of $\beta_{k+1}$. By \eqref{eq:7} and the assumption that $\sigma^2<1/2$ we conclude our claim. 

Now, we observe that since the sequence $(\eta_k-t_k(f+g)(\overline{x}^k))_{k\in\N}$ is non-decreasing, we have, for all $k\geq1$, that 
\begin{equation*}
\eta_k - t_k(f+g)(\overline{x}^k) \geq \eta_0 - t_0(f+g)(\overline{x}^0) = 0.
\end{equation*} 
Hence, using \eqref{eq:6} we deduce the second inequality in \eqref{eq:9}.

To prove the first inequality in \eqref{eq:9}, we note that from the definitions of $t_k$ and $\beta_k$ it follows that
\begin{equation*}
t_k \geq t_{k-1} + \dfrac{\alpha(1-\sigma^2)}{2} + \sqrt{\alpha(1-\sigma^2)t_{k-1}} \geq (\sqrt{t_{k-1}} +\dfrac12\sqrt{\alpha(1-\sigma^2)} )^2.
\end{equation*} 
Thus, we conclude using that $\alpha=\sigma^2/L$ and an induction argument. \hfill$\qed$

\end{document}